\documentclass[11pt,oneside]{amsart}
\usepackage{fouriernc} 

\usepackage{Environments}   
\usepackage{Definitions}    
\usepackage{PageSetup}      

\makeatletter
\newcommand{\xRightarrow}[2][]{\ext@arrow 0359\Rightarrowfill@{#1}{#2}}
\makeatother

\usepackage{graphicx}
\newcommand{\Anglearrow}[1]{\rotatebox{#1}{$\Rightarrow$}}

\hypersetup{
 pdfkeywords={Postnikov system,k-invariant,symmetric monoidal
   2-category,K-theory spectrum,Picard 2-category,2-monad},
 pdfauthor={Nick Gurski,Niles Johnson,Angélica Osorno,Marc Stephan},
}


\usepackage{AuthorInfo}     

\title{Stable Postnikov data of Picard 2-categories}
\authorGurski
\authorJohnson
\authorOsorno
\authorStephan
\date{22 February, 2017}

\theoremstyle{plain}
\newtheorem*{shh}{Stable Homotopy Hypothesis}
\newtheorem*{hh}{Homotopy Hypothesis}

\begin{document}

\begin{abstract}
Picard 2-categories are symmetric monoidal 2-categories with
invertible 0-, 1-, and 2-cells. The classifying space of a Picard
2-category $\mathcal{D}$ is an infinite loop space, the zeroth space of the
$K$-theory spectrum $K\mathcal{D}$. This spectrum has stable homotopy groups
concentrated in levels 0, 1, and 2. In this paper, we describe part
of the Postnikov data of $K\mathcal{D}$ in terms of categorical structure.
We use this to show that there is no strict skeletal Picard
2-category whose $K$-theory realizes the 2-truncation of the sphere
spectrum. As part of the proof, we construct a categorical
suspension, producing a Picard 2-category $\Sigma C$ from a Picard
1-category $C$, and show that it commutes with $K$-theory in that
$K\Sigma C$ is stably equivalent to $\Sigma K C$.
\end{abstract}

\maketitle


\section{Introduction}
This paper is part of a larger effort to refine and expand the theory
of algebraic models for homotopical data, especially that of
\emph{stable} homotopy theory.  Such modeling has been of interest
since \cite{May1974Einfty,Seg74Categories} gave $K$-theory functors which
build connective spectra from symmetric monoidal categories.
Moreover, Thomason \cite{Tho95Symmetric} proved that symmetric
monoidal categories have a homotopy theory which is equivalent to that
of all connective spectra.  

Our current work is concerned with constructing models for stable
homotopy 2-types using symmetric monoidal 2-categories.  Preliminary
foundations for this appear, for example, in \cite{ GO12Infinite,
  GJO2015KTheory, JO12Modeling, Sch2011Classification}.  In
forthcoming work \cite{GJOshh2} we prove that all stable homotopy
2-types are modeled by a special kind of symmetric monoidal
2-categories which we describe below and call \emph{strict Picard
  2-categories}.

Research leading to the methods in \cite{GJOshh2} has shown
that the most difficult aspect of this problem is replacing a
symmetric monoidal 2-category modeling an arbitrary connective
spectrum (see \cite{GJO2015KTheory}) by a strict Picard 2-category with
the same stable homotopy 2-type.  This paper can then be interpreted
as setting a minimum level of complexity for such a categorical model
of stable homotopy 2-types.  Furthermore, we intend to construct the
Postnikov tower for a stable homotopy 2-type entirely within a
categorical context, and the results here give some guidance as to the
assumptions we can make on those Postnikov towers.

This paper has three essential goals.  First, we explicitly describe
part of the Postnikov tower for strict Picard 2-categories.  Second,
and of independent interest, we show that the $K$-theory functor
commutes with suspension up to stable equivalence.  This allows us to
bootstrap previous results on Picard 1-categories to give algebraic
formulas for the two nontrivial Postnikov layers of a Picard
2-category. Third, we combine these to show that, while strict Picard
2-categories are expected to model all stable homotopy 2-types,
strict and \emph{skeletal} Picard 2-categories cannot.  We prove that
there is no strict and skeletal Picard 2-category modeling the
truncation of the sphere spectrum.

\subsection{Background and motivation}

Homotopical invariants, and therefore homotopy \emph{types}, often
have a natural interpretation as categorical structures.  The
fundamental groupoid is a complete invariant for homotopy 1-types,
while pointed connected homotopy 2-types are characterized by their
associated crossed module or $\cat^1$-group structure
\cite{Whi1949Combinatorial,MW19503type,BS76Ggroupoids,Loday82,conduche84}.
Such characterizations provide the low-dimensional cases of
Grothendieck's \emph{Homotopy Hypothesis}
\cite{grothendieck1983pursuing}.

\begin{hh}
  There is an equivalence of homotopy theories between
  $\mathpzc{Gpd}^n$, weak $n$-groupoids equipped with categorical
  equivalences, and $\mathpzc{Top}^n$, homotopy $n$-types equipped
  with weak homotopy equivalences.
\end{hh}

Restricting attention to stable phenomena, we replace homotopy
$n$-types with stable homotopy $n$-types: spectra $X$ such that $\pi_i
X = 0$ unless $0 \leq i \leq n$.  On the categorical side, we take a
cue from \cite{May1974Einfty,Tho95Symmetric} and replace $n$-groupoids with
a grouplike, symmetric monoidal version that we call Picard
$n$-categories.  The stable version of the Homotopy Hypothesis is then
the following.
\begin{shh}
  There is an equivalence of homotopy theories between
  $\mathpzc{Pic}^n$, Picard $n$-categories equipped with categorical
  equivalences, and $\mathpzc{Sp}_0^n$, stable homotopy $n$-types
  equipped with stable equivalences.
\end{shh}

For $n=0$, $\mathpzc{Pic}^0$ is the category of abelian groups
$\mathpzc{Ab}$ with weak equivalences given by group isomorphisms.  It
is equivalent to the homotopy theory of Eilenberg-Mac Lane
spectra. For $n=1$, a proof of the Stable Homotopy Hypothesis appears
in \cite{JO12Modeling}, and a proof for $n=2$ will appear in the
forthcoming \cite{GJOshh2}.  The advantage of being able to work with
categorical weak equivalences is that the maps in the homotopy
category between two stable 2-types modeled by strict Picard
2-categories are realized by symmetric monoidal pseudofunctors between
the two strict Picard 2-categories, instead of having to use general
zigzags. In fact, as will appear in \cite{GJOshh2}, the set of
homotopy classes between two strict Picard 2-categories $\cD$ and
$\cD'$ is the quotient of the set of symmetric monoidal pseudofunctors
$\cD\to \cD'$ by the equivalence relation $F\sim G$ if there exists a
pseudonatural transformation $F\Rightarrow G$.

More than a proof of the Stable Homotopy Hypothesis, we seek a
complete dictionary translating between stable homotopical invariants
and the algebra of Picard $n$-categories.  The search for such a
dictionary motivated three questions that lie at the heart of this
paper.  First, how can we express invariants of stable homotopy types
in algebraic terms?  Second, how can we construct stable homotopy
types of interest, such as Postnikov truncations of the sphere
spectrum, from a collection of invariants?  Third, can we make
simplifying assumptions, such as strict inverses, about Picard
$n$-categories without losing homotopical information?

The results in this paper provide key steps toward answering these
questions. In particular, we characterize the three stable homotopy
groups of a strict Picard 2-category in terms of equivalence classes
of objects, isomorphism classes of 1-cells, and 2-cells, respectively,
and deduce that a map of Picard 2-categories is a stable equivalence
if and only if it is a categorical equivalence
(\cref{prop:st-equiv-cat-equiv}). This fact is used in \cite{GJOshh2}
to prove the Stable Homotopy Hypothesis for $n = 2$.

\subsection{Postnikov invariants and strict skeletalization}

It has long been folklore that the symmetry in a Picard 1-category
should model the bottom $k$-invariant, $k_0$.  Along with a proof of
the Stable Homotopy Hypothesis in dimension 1, this folklore result was
established in \cite{JO12Modeling}.  
This shows that a Picard
1-category is characterized by exactly three pieces of data: an abelian
group of isomorphism classes of objects $(\pi_0)$, an abelian group of
automorphisms of the unit object $(\pi_1)$, and a group homomorphism
$k_0 \cn \pi_0 \otimes \mathbb{Z}/2 \to \pi_1$ (i.e., a stable
quadratic map from $\pi_0$ to $\pi_1$) corresponding to the symmetry.
Such a characterization is implied by the following result.
\begin{thm}[{\cite[Theorem
    2.2]{JO12Modeling}}]\label{thm:strskel_pic}
  Every Picard category is equivalent to one which is both strict and
  skeletal.
\end{thm}
We call this phenomenon \emph{strict skeletalization}.  This theorem
is quite surprising given that it is false without the symmetry.
Indeed, Baez and Lauda \cite{BL2004Higher} give a good account of the
failure of strict skeletalization for 2-groups (the non-symmetric
version of Picard 1-categories), and how it leads to a cohomological
classification for 2-groups.  Johnson and Osorno \cite{JO12Modeling}
show, in effect, that the relevant obstructions are \emph{unstable}
phenomena which become trivial upon stabilization.

When we turn to the question of building models for specific homotopy
types, the strict and skeletal ones are the simplest: given a stable
1-type $X$, a strict and skeletal model will have objects equal to the
elements of $\pi_0 X$ and automorphisms of every object equal to
the elements of $\pi_1 X$, with no morphisms between distinct objects.
All that then remains is to define the correct symmetry isomorphisms,
and these are determined entirely by the map $k_0$.

As an example, a strict and skeletal model for the 1-truncation of the
sphere spectrum has objects the integers, each hom-set of
automorphisms the integers mod 2, and $k_0$ given by the identity map
on $\mathbb{Z}/2$ corresponding to the fact that the generating object
$1$ has a nontrivial symmetry with itself.  One might be tempted to build a strict
and skeletal model for the 2-type of the sphere spectrum (the authors
here certainly were, and such an idea also appears in
\cite[Example 5.2]{Bar14Quasistrict}).  But here we prove that this is not
possible for the sphere spectrum, and in fact a large class of stable
2-types.

\begin{thm}[\cref{thm:nostrskelmod}]
\label{thm:nostrskelmodintro}
  Let $\cD$ be a strict skeletal Picard 2-category with $k_0$
  surjective.  Then the 0-connected cover of $K\cD$ splits as a
  product of Eilenberg-Mac Lane spectra.  
  In particular, there is no strict and skeletal model
  of the 2-truncation of the sphere spectrum.
\end{thm}

Our proof of this theorem identifies both the bottom $k$-invariant
$k_0$ and the first Postnikov layer $k_1i_1$ (see
\cref{sec:hty-thy-Pic-cats}) of $K\cD$ explicitly using the symmetric
monoidal structure for any strict Picard 2-category
$\cD$. In addition, we provide a categorical
model of the 1-truncation of $K\cD$ in \cref{lem:1-truncation}. This
provides data which is necessary, although not sufficient, for a
classification of stable 2-types akin to the cohomological
classification in \cite{BL2004Higher}. Remaining data, to be studied
in future work, must describe the connection of $\pi_2$ with $
\pi_0$. For instance, stable 2-types $X$ with trivial $\pi_1$ are
determined by a map $H(\pi_0 X)\to \Sigma^3 H(\pi_2 X)$ in the stable
homotopy category.  For general $X$, the third cohomology group of the
1-truncation of $X$ with coefficients in $\pi_2X$ has to be
calculated. In the spectral sequence associated to the stable
Postnikov tower of $X$ (see \cite[Appendix B]{GM95Generalized}), the
connection between $\pi_0$ and $\pi_2$ becomes apparent in the form of
a $d_3$ differential.

In addition to clarifying the relationship between Postnikov
invariants and the property of being skeletal,
\cref{thm:nostrskelmodintro} suggests a direction for future work
developing a 2-categorical structure that adequately captures the
homotopy theory of stable 2-types.  Such structure ought to be more
specific than that of strict Picard 2-categories but more general than
strict, skeletal Picard 2-categories.  Interpretations of this
structure which are conceptual (in terms of other categorical
structures) and computational (in terms of homotopical or homological
invariants, say) will shed light on both the categorical and
topological theory.

\subsection{Categorical suspension}
In order to give a formula for the first Postnikov layer,
we must show that $K$-theory functors are compatible with
suspension. More precisely, given a strict monoidal category $C$, one
can construct a one-object 2-category $\Si C$, where the category of
morphisms is given by $C$, with composition defined using the monoidal
structure. Further, if $C$ is a permutative category then $\Si C$ is
naturally a symmetric monoidal 2-category, with the monoidal structure
also defined using the structure of $C$. Unstably, it is known that
this process produces a categorical delooping: if $C$ is a strict
monoidal category with invertible objects, the classifying space
$B(\Si C)$ is a delooping of $BC$ \cite{Jar91super,CCG10Nerves}. We prove 
the stable analogue.
\begin{thm*}[\cref{thm:sigmak_equals_ksigma}]
  For any permutative category $C$, the spectra $K(\Si C)$ and $\Si
  (KC)$ are stably equivalent.
\end{thm*}
Here $K(-)$ denotes both the $K$-theory spectrum associated to a
symmetric monoidal category \cite{May1974Einfty,Seg74Categories} and the
$K$-theory spectrum associated to a symmetric monoidal 2-category
\cite{GO12Infinite,GJO2015KTheory}.

This theorem serves at least three purposes beyond being a necessary
calculation tool.  A first step in the proof is \cref{cor:equivs}
which shows that the categories of permutative categories and of
one-object permutative Gray-monoids are equivalent; this is a strong
version of one case of the Baez-Dolan Stabilization Hypothesis
\cite{BD1998Higher}, stronger than the usual proofs in low dimensions
\cite{CG2007periodicI,CG2011periodicII,CG2014Iterated}.
The second purpose of this theorem is to justify, from a homotopical
perspective, the definition of permutative Gray-monoid, the
construction of the $K$-theory spectrum, and the categorical
suspension functor.  The suspension functor of spectra and the
$K$-theory spectrum of a permutative category are both central
features of stable homotopy theory, so any generalization of the
latter should respect the former.  A final purpose of this theorem
will appear in future work, namely in the categorical construction of
stable Postnikov towers.  Suspension spectra necessarily appear in
these towers, and \cref{thm:sigmak_equals_ksigma} and \cref{cor:equivs}
together allow us to replicate these features of a Postnikov tower
entirely within the world of symmetric monoidal 2-categories.

\subsection{Relation to supersymmetry and supercohomology}

The theory of Picard 2-categories informs recent work in mathematical
physics related to higher supergeometry \cite{Kap2015Supergeometry}
and invertible topological field theories
\cite{freed2014anomalies}.   In
\cite{Kap2015Supergeometry}, Kapranov links the $\bZ$-graded Koszul
sign rule appearing in supergeometry to the 1-truncation of the sphere
spectrum. He describes how higher supersymmetry is governed by higher
truncations of the sphere spectrum, which one expects to be modeled by
the free Picard $n$-category on a single object.  Likewise, Freed
\cite{freed2014anomalies} describes examples using the Picard
bicategory of complex invertible super algebras related to twisted
$K$-theory \cite{FHT2011loop}.

The failure of strict skeletalization for a categorical model of the
2-truncation of the sphere spectrum shows that already for $n=2$
capturing the full higher supersymmetry in algebraic terms is more
complicated than one might expect.

Furthermore, it would be interesting to relate examples appearing in
physics literature about topological phases of matter
\cite{gu2014symmetry,bhardwaj2016state} to cohomology with
coefficients in Picard $n$-categories. The \emph{super}-cohomology in
\cite{gu2014symmetry} is assembled from two different classical
cohomology groups of a classifying space $BG$ with a nontrivial
symmetry. One expects that this super-cohomology can be expressed as
the cohomology of $BG$ with coefficients in a Picard 1-category, and
similarly, for the extension of this super-cohomology in
\cite{bhardwaj2016state} as cohomology with coefficients in a Picard
2-category.

\subsection*{Outline}
In \cref{sec:pic-2-cat} we sketch the basic theory of Picard
categories and Picard 2-categories.  This includes some background to
fix notation and some recent results about symmetric monoidal
2-categories \cite{GJO2015KTheory}.  In \cref{sec:hty-thy-Pic-cats} we
develop algebraic models for some of the Postnikov data of the
spectrum associated to a Picard 2-category, giving formulas for the
two nontrivial layers in terms of the symmetric monoidal structure.
This section closes with applications showing that strict skeletal
Picard 2-categories cannot model all stable 2-types.
\Cref{sec:strictification} 
establishes formal strictification results for 2-categorical diagrams using 2-monad theory.  We use
those results in
\cref{sec:suspension-models-suspension} to prove that the $K$-theory
functor commutes with suspension.

\subsection*{Acknowledgements}
The authors gratefully acknowledge travel funding and hospitality from
The Ohio State University Mathematics Research Institute. The first
named author was supported by EPSRC EP/K007343/1, and the fourth by
SNSF grants 158932 and 151967. This material is partially based on
work supported by the National Science Foundation under Grant
No. 0932078 000 while the third and fourth named authors were in
residence at the Mathematical Sciences Research Institute in Berkeley,
California, during the Spring 2014 semester. This work was partially
supported by a grant from the Simons Foundation (\#359449,
Ang\'{e}lica Osorno).

The authors thank Jack Morava for encouraging them to explore the
relations to supersymmetry.  They also thank the referee whose
comments helped improve the exposition.

\section{Picard categories and Picard 2-categories}
\label{sec:pic-2-cat}

This section introduces the primary categorical structures of interest
which we call Picard 2-categories, as well as the particularly
relevant variant of strict skeletal Picard 2-categories.  Note that
we use the term 2-category in its standard sense \cite{KS74Review},
and in particular all composition laws are strictly associative and
unital.

\begin{notn}\label{notn:cat-iicat}
  We let $\cat$ denote the category of categories and functors, and
  let $\iicat$ denote the category of 2-categories and 2-functors.
  Note that these are both 1-categories.
\end{notn}

\begin{notn}\label{notn:cat2-iicat2}
We let $\cat_2$ denote the 2-category of categories, functors, and
natural transformations.  This can be thought
of as the 2-category of categories enriched in $\set$. Similarly, we
let $\iicat_2$ denote the 2-category of 2-categories, 2-functors and 2-natural
transformations; the 2-category of categories
enriched in $\cat$. 
\end{notn}

\subsection{Picard categories}

We will begin by introducing all of the 1-categorical notions before
going on to discuss their 2-categorical analogues.  First we recall
the notion of a permutative category (i.e., symmetric strict monoidal
category); the particular form of this definition allows an easy
generalization to structures on 2-categories.

\begin{defn}\label{defn:pc}
  A \textit{permutative category} $C$ consists of a strict monoidal
  category $(C, \oplus, e)$ together with a natural isomorphism,
  \[
  \xy
  (0,0)*+{C \times C}="00";
  (25,0)*+{C \times C}="10";
  (12.5,-10)*+{C}="01";
  {\ar^{\tau} "00"; "10"};
  {\ar^{\oplus} "10"; "01"};
  {\ar_{\oplus} "00"; "01"};
  (12.5,-4)*{\Anglearrow{40} \beta}
  \endxy
  \]
  where $\tau \cn C \times C \to C \times C$ is the symmetry
  isomorphism in $\Cat$, such that the following axioms hold for all
  objects $x,y,z$ of $C$.
  \begin{itemize}
  \item $\beta_{y,x} \beta_{x,y} = \id_{x \oplus y}$
  \item $\beta_{e,x} = \id_{x} = \beta_{x,e}$
  \item $\beta_{x, y \oplus z} = (y \oplus \beta_{x,z}) \circ (\beta_{x,y} \oplus z)$
  \end{itemize}
\end{defn}

\begin{rmk}\label{rmk:strictsmc}
  We will sometimes say that a symmetric monoidal structure on a
  category is strict if its underlying monoidal structure is.  Note
  that this does not imply that the symmetry is the identity, even
  though the other coherence isomorphisms are.  Thus a permutative
  category is nothing more than a strict symmetric monoidal category.
\end{rmk}

\begin{notn}\label{notn:pcat}
  Let $\pcat$ denote the category of permutative categories and
  symmetric, strict monoidal functors between them.
\end{notn}

Next we require a notion of invertibility for the objects in a
symmetric monoidal category.

\begin{defn}\label{defn:invertible1}
  Let $(C, \oplus, e)$ be a monoidal category.  An object $x$ is
  \emph{invertible} if there exists an object $y$ together with
  isomorphisms $x \oplus y \cong e$, $y \oplus x \cong e$.
\end{defn}

\begin{defn}\label{defn:pic1}
  A \emph{Picard category} is a symmetric monoidal category in which
  all of the objects and morphisms are invertible.
\end{defn}

The terminology comes from the following example.
\begin{example}
  Let $R$ be a commutative ring, and consider the symmetric monoidal
  category of $R$-modules.  We have the subcategory $\pic R$ of
  invertible $R$-modules and isomorphisms between them.  The set
  of isomorphism classes of objects of $\pic R$ is the classical
  Picard group of $R$.
\end{example}

\begin{rmk}
  If we drop the symmetric structure in \cref{defn:pic1} above, we get
  the notion of what is both called a categorical group
  \cite{JS1993btc} or a 2-group \cite{BL2004Higher}.  These are
  equivalent to crossed modules \cite{Whi1949Combinatorial,Loday82}, and hence
  are a model for pointed connected homotopy 2-types (i.e., spaces $X$ for
  which $\pi_{i}(X) = 0$ unless $i = 1, 2$).
\end{rmk}

One should consider Picard categories as a categorified version of
abelian groups.  Just as abelian groups model the homotopy theory of
spectra with trivial homotopy groups aside from $\pi_0$, Picard
categories do the same for spectra with trivial homotopy groups aside
from $\pi_0$ and $\pi_1$.

\begin{thm}[{\cite[Theorem
    1.5]{JO12Modeling}}]\label{thm:pic1_model_1types}
  There is an equivalence of homotopy theories between the category of
  Picard categories, $\mathpzc{Pic}^1$, equipped with categorical
  equivalences, and the category of stable 1-types,
  $\mathpzc{Sp}_0^1\,$, equipped with stable equivalences.
\end{thm}

Forthcoming work \cite{GJOshh2} proves the 2-dimensional analogue
of \cref{thm:pic1_model_1types}.  This requires a theory of Picard
2-categories which began in \cite{GJO2015KTheory} and motivated the
work of the current paper.  We now turn to such theory.

\subsection{Picard 2-categories}\label{sec:pic-2-cats}

To give the correct 2-categorical version of Picard categories, we
must first describe the analogue of a mere strict monoidal category:
such a structure is called a Gray-monoid.  It is most succinctly
defined using the Gray tensor product of 2-categories, written $\cA
\otimes \cB$ for a pair of 2-categories $\cA$, $\cB$. We will not give
the full definition of $\otimes$ here (see
\cite{GJO2015KTheory,Gurski13Coherence,BG2015cocategorical,BG2015Gray})
but instead give the reader the basic idea.  The objects of $\cA
\otimes \cB$ are tensors $a \otimes b$ for $a \in \cA, b \in \cB$, but
the 1-cells are \emph{not} tensors of 1-cells as one would find in the
cartesian product.  Instead they are generated under composition by
1-cells $f \otimes 1$ and $1 \otimes g$ for $f\cn a \to a'$ a 1-cell
in $\cA$ and $g\cn b \to b'$ a 1-cell in $\cB$.  These different kinds
of generating 1-cells do not commute with each other strictly, but
instead up to specified isomorphism 2-cells
\[
\Si_{f,g} \cn (f \otimes 1)\circ(1 \otimes g) \cong (1 \otimes g)\circ(f \otimes 1)
\]
which obey appropriate naturality and bilinearity axioms.  We call
these $\Si$ the \emph{Gray structure 2-cells}.  The 2-cells of $\cA
\otimes \cB$ are defined similarly, generated by $\al \otimes 1, 1
\otimes \be$, and the $\Si_{f,g}$.  The function $(\cA, \cB) \mapsto
\cA \otimes \cB$ is the object part of a functor of categories
\[
\iicat \times \iicat \to \iicat
\]
which is the tensor product for a symmetric monoidal structure on
$\iicat$ with unit the terminal 2-category.

\begin{defn}\label{defn:gm}
  A \emph{Gray-monoid} is a monoid object $(\cD, \oplus, e)$ in the
  monoidal category $(\iicat, \otimes)$.
\end{defn}

\begin{rmk}
  By the coherence theorem for monoidal bicategories
  \cite{GPS95Coherence,Gurski13Coherence}, every monoidal bicategory
  is equivalent (in the appropriate sense) to a Gray-monoid.  There is
  a stricter notion, namely that of a monoid object in $(\iicat,
  \times)$, but a general monoidal bicategory will not be equivalent
  to one of these.
\end{rmk}

We now turn to the symmetry.

\begin{defn}\label{defn:pgm}
  A \textit{permutative Gray-monoid} $\cD$ consists of a Gray-monoid
  $(\cD, \oplus, e)$ together with a 2-natural isomorphism,
  \[
  \xy
  (0,0)*+{\cD \otimes \cD}="00";
  (25,0)*+{\cD \otimes \cD}="10";
  (12.5,-10)*+{\cD}="01";
  {\ar^{\tau} "00"; "10"};
  {\ar^{\oplus} "10"; "01"};
  {\ar_{\oplus} "00"; "01"};
  (12.5,-4)*{\Anglearrow{40}\beta}
  \endxy
  \]
  where $\tau \cn \cD \otimes \cD \to \cD \otimes \cD$ is the symmetry
  isomorphism in $\IICat$ for the Gray tensor product, such that the
  following axioms hold.
  \begin{itemize}
  \item The following pasting diagram is equal to the identity
    2-natural transformation for the 2-functor $\oplus$.
    \[
    \xy
    (0,0)*+{\cD \otimes \cD}="00";
    (25,0)*+{\cD \otimes \cD}="10";
    (50,0)*+{\cD \otimes \cD}="20";
    (25,-15)*+{\cD}="11";
    {\ar^{\tau} "00"; "10"};
    {\ar^{\tau} "10"; "20"};
    {\ar_{\oplus} "00"; "11"};
    {\ar_{\oplus} "10"; "11"};
    {\ar^{\oplus} "20"; "11"};
    {\ar@/^1.5pc/^1 "00"; "20"};
    (14.5,-4)*{\scriptstyle \Anglearrow{40} \beta};
    (35.5,-4)*{\scriptstyle \Anglearrow{0} \beta}
    \endxy
    \]
  \item The following pasting diagram is equal to the identity
    2-natural transformation for the canonical isomorphism
    $1 \otimes \cD \cong \cD$.
    \[ \xy
    (0,0)*+{1 \otimes \cD}="00";
    (30,0)*+{\cD \otimes \cD}="10";
    (60,0)*+{\cD \otimes \cD}="20";
    (30,-12)*+{\cD}="11";
    {\ar^{e \otimes \id} "00"; "10"};
    {\ar^{\tau} "10"; "20"};
    {\ar^{\oplus} "20"; "11"};
    {\ar_{\oplus} "10"; "11"};
    {\ar_{\cong} "00"; "11"};
    (17,-4)*{\scriptstyle =}; 
    (40, -4)*{\scriptstyle \Anglearrow{0} \beta}
    \endxy
    \]
  \item The following equality of pasting diagrams holds where we have
    abbreviated the tensor product to concatenation when labeling 1-
    or 2-cells.
    \[ 
    \xy
    0;<.95mm,0mm>:<0mm,1mm>:: 
    (3,-10)*+{\cD^{\otimes 3}}="00";
    (18,0)*+{\cD^{\otimes 3}}="10";
    (36,0)*+{\cD^{\otimes 3}}="20";
    (51,-10)*+{\cD^{\otimes 2}}="30";
    (27,-15)*+{\cD^{\otimes 2}}="11";
    (18,-30)*+{\cD^{\otimes 2}}="12";
    (36,-30)*+{\cD}="33";
    (69,-10)*+{\cD^{\otimes 3}}="40";
    (84,0)*+{\cD^{\otimes 3}}="50";
    (102,0)*+{\cD^{\otimes 3}}="60";
    (117,-10)*+{\cD^{\otimes 2}}="70";
    (84,-30)*+{\cD^{\otimes 2}}="52";
    (102,-30)*+{\cD}="73";
    (102,-17)*+{\cD^{\otimes 2}}="63";
    {\ar^{\scriptstyle \tau \id} "00"; "10"};
    {\ar^{\scriptstyle \tau \id} "40"; "50"};
    {\ar^{\scriptstyle \id \tau } "10"; "20"};
    {\ar^{\scriptstyle \id \tau } "50"; "60"};
    {\ar^{\scriptstyle \oplus \id} "20"; "30"};
    {\ar^{\scriptstyle \oplus \id} "60"; "70"};
    {\ar^{\scriptstyle \oplus} "30"; "33"};
    {\ar^{\scriptstyle \oplus} "70"; "73"};
    {\ar_{\scriptstyle \oplus \id} "00"; "12"};
    {\ar_{\scriptstyle \oplus \id} "40"; "52"};
    {\ar_{\scriptstyle \oplus} "12"; "33"};
    {\ar_{\scriptstyle \oplus} "52"; "73"};
    {\ar_{\scriptstyle \id \oplus} "00"; "11"};
    {\ar^{\scriptstyle \tau} "11"; "30"};
    {\ar_{\scriptstyle \oplus} "11"; "33"};
    {\ar^{\scriptstyle \oplus \id} "50"; "52"};
    {\ar_{\scriptstyle \id \oplus} "50"; "63"};
    {\ar^{\scriptstyle \id \oplus} "60"; "63"};
    {\ar^{\scriptstyle \oplus} "63"; "73"};
    (27,-7.5)*{=}; (19,-21)*{=}; (108,-12)*{=}; (93,-20)*{=}; (59,-15)*{=};
    (37,-19)*{\scriptstyle \Anglearrow{40} \beta}; 
    (78.4,-12.5)*{\scriptstyle \Anglearrow{40} \beta \id}; 
    (96,-5)*{\scriptstyle \Anglearrow{40} \id \beta};
    \endxy
    \]
  \end{itemize}
\end{defn}
\begin{rmk}\label{defn:sm2cat}
  A symmetric monoidal 2-category is a symmetric monoidal bicategory
  (see \cite{GJO2015KTheory} for a sketch or \cite{McCru00Balanced}
  for full details) in which the underlying bicategory is a
  2-category. Every symmetric monoidal bicategory is equivalent as
  such to a symmetric monoidal 2-category by strictifying the
  underlying bicategory and transporting the structure as in
  \cite{Gur2012Biequivalences}.  A deeper result is that every
  symmetric monoidal bicategory is equivalent as such to a permutative
  Gray-monoid; this is explained fully in \cite{GJO2015KTheory},
  making use of \cite{Sch2011Classification}.
\end{rmk}

\begin{notn}
  For convenience and readability, we use following notational conventions for cells in a Gray-monoid $\cD$.
  \begin{itemize}
  \item For objects, we may use concatenation instead of explicitly
    indicating the monoidal product.
  \item For an object $b$ and a 1-cell $f\colon a \to a'$, we denote
    by $fb$ the 1-cell in $\cD$ which is the image under $\oplus$ of
    $f\otimes 1\colon a\otimes b \to a'\otimes b$ in $\cD \otimes
    \cD$.  We use similar notation for multiplication on the other
    side, and for 2-cells.
  \item We let $\Sigma_{f,g}$ also denote the image in $\cD$ of the
    Gray structure 2-cells under $\oplus$:
    \[
    \Sigma_{f,g} \cn (fb')\circ(ag) \cong (a' g)\circ(f b).
    \]
  \end{itemize}
\end{notn}

\begin{notn}
  Let $\pgm$ denote the category of permutative Gray-monoids and
  strict symmetric monoidal 2-functors between them.
\end{notn}

We are actually interested in permutative Gray-monoids which model
stable homotopy 2-types, and we therefore restrict to those in which
all the cells are invertible.  We begin by defining invertibility in a
Gray-monoid, then the notion of a Picard 2-category, and finish with
that of a strict skeletal Picard 2-category.

\begin{defn}\label{defn:invertible2}
  Let $(\cD, \oplus, e)$ be a Gray-monoid.
  \begin{enumerate}
  \item A 2-cell of $\cD$ is invertible if it has an inverse in the
    usual sense.
  \item A 1-cell $f \cn x \to y$ is invertible if there exists a
    1-cell $g \cn y \to x$ together with invertible 2-cells
    $g\circ f \cong \id_{x}$, $f\circ g \cong \id_{y}$.  In other words, $f$ is
    invertible if it is an internal equivalence (denoted with the
    $\simeq$ symbol) in $\cD$.
  \item An object $x$ of $\cD$ is invertible if there exists another
    object $y$ together with invertible 1-cells $x \oplus y \simeq e$,
    $y \oplus x \simeq e$.
  \end{enumerate}
\end{defn}

\begin{rmk}
  The above definition actually used none of the special structure of
  a Gray-monoid that is not also present in a more general monoidal
  bicategory.
\end{rmk}

\begin{defn}\label{defn:pic2}
  A \emph{Picard 2-category} is a symmetric monoidal 2-category
  (see \cref{defn:sm2cat}) in which all of the objects, 1-cells, and
  2-cells are invertible.  A \emph{strict Picard 2-category} is a
  permutative Gray-monoid which is a Picard 2-category.
\end{defn}

\begin{rmk}
  Note that the definition of a strict Picard 2-category does not
  require that cells be invertible in the strict sense, i.e., having
  inverses on the nose rather than up to mediating higher cells.  It
  only requires that the underlying symmetric monoidal structure is
  strict in the sense of being a permutative
  Gray-monoid.  
\end{rmk}

\begin{defn}\label{defn:skeletal2}
  A 2-category $\cA$ is \emph{skeletal} if the following condition
  holds: whenever there exists an invertible 1-cell $f\cn x \simeq y$,
  then $x=y$.
\end{defn}

\begin{rmk}
  This definition might more accurately be named \emph{skeletal on
    objects}, as one could impose a further condition of being
  skeletal on 1-cells as well.  We have no need of this further
  condition, and so we work with this less restrictive notion of a
  skeletal 2-category.  It is also important to remember that, in the
  definition above, the invertible 1-cell $f$ need not be the identity
  1-cell.  The slogan is that ``every equivalence is an
  autoequivalence'': an object is allowed to have many non-identity
  autoequivalences, and there can be 1-cells between different objects
  as long as they are not equivalences.
\end{rmk}

\begin{defn}\label{defn:strskel_pic2}
  A \emph{strict skeletal Picard 2-category} is a strict Picard
  2-category whose underlying 2-category is skeletal.
\end{defn}

\subsection{Two adjunctions}\label{sec:two-adjunctions}
Our goal in this subsection is to present two different adjunctions
between strict Picard categories and strict Picard 2-categories.
While we focus on the categorical algebra here, later we will give
each adjunction a homotopical interpretation.  The unit of the first
adjunction will categorically model Postnikov 1-truncation
(\cref{lem:1-truncation}), universally making $\pi_2$ zero, while the
counit of the second will categorically model the 0-connected cover
(\cref{prop:susp-conn-cover}).

Recall that for any category $C$, we have its set of path components
denoted $\pi_0 C$; these are given by the path components of the nerve
of $C$, or equivalently by quotienting the set of objects by the
equivalence relation generated by $x \sim y$ if there exists an arrow
$x \to y$.  This is the object part of a functor $\pi_0 \cn \cat \to
\set$, and it is easy to verify that this functor preserves finite
products.  It is also left adjoint to the functor $d \cn \set \to
\cat$ which sends a set $S$ to the discrete category with the same set
of objects.  Being a right adjoint, $d$ preserves all products.  The
counit $\pi_0 \circ d \Rightarrow \id$ is the identity, and the unit
$\id \Rightarrow d \circ \pi_0$ is the quotient functor $C \to d\pi_0
C$ sending every object to its path component and every morphism to
the identity.
Since $d$ and $\pi_0$ preserve products, by
applying them to hom-objects they induce change of enrichment functors
$d_*$ and $(\pi_0)_*$, respectively. We obtain the following result.

\begin{lem}\label{lem:pi0_lifts}
The adjunction $\pi_0 \dashv d$ lifts to a 2-adjunction 
\[
 \begin{xy}
    0;<15mm,0mm>:<0mm,10mm>:: 
    (-1,0)*+{\iicat_2}="2cat";
    (1,0)*+{\cat_2.}="cat";
    {\ar@/^3mm/^-{(\pi_0)_*} "2cat"; "cat"};
    {\ar@/^3mm/^-{d_*} "cat"; "2cat"};
    (0,0)*+{\bot}="v";
  \end{xy}
\]
\end{lem}

\begin{notn}\label{pi0*=_1}
  We will write the functor $(\pi_0)_{*}$ as $\cD \mapsto \cD_1$ to
  lighten the notation.  This anticipates the homotopical
  interpretation in \cref{lem:1-truncation}.  Furthermore, we will
  write $d_*$ as $d$, it will be clear from context which functor we
  are using.
\end{notn}

\begin{lem}\label{prop:_1ssm}
  The functor $\cD \mapsto \cD_1$ is strong symmetric monoidal
  $(\iicat, \otimes) \to (\cat, \times)$.  The functor $d$ is lax
  symmetric monoidal $(\cat, \times) \to (\iicat, \otimes)$.
\end{lem}
\begin{proof}
  The second statement follows from the first by doctrinal adjunction
  \cite{Kelly1974Doctrinal}.  For the first, one begins by checking
  that
  \[
  \cD_1 \times \cE_1 \cong (\cD \otimes \cE)_{1};
  \]
  this is a simple calculation using the definition of $\otimes$ that
  we leave to the reader.  If we let $I$ denote the terminal
  2-category, the unit for $\otimes$, then $I_1$ is the terminal
  category, so $(-)_1$ preserves units up to (unique) isomorphism.  It
  is then easy to check that these isomorphisms interact with the
  associativity, unit, and symmetry isomorphisms to give a strong
  symmetric monoidal functor.
\end{proof}

\begin{rmk}\label{rmk:d_*monoidal}
  It is useful to point out that if $A,B$ are categories, then the
  comparison 2-functor
  \[
  \chi_{A,B} \cn d A \otimes d B \to d (A \times B)
  \]
  is the 2-functor which quotients all the 2-cells $\Sigma_{f,g}$ to
  be the identity.  In view of the adjunction in \cref{lem:pi0_lifts},
  the 2-functor $\chi_{A,B}$ can be identified with the component of
  the unit at $d A \otimes d B$.
\end{rmk}

Our first adjunction between Picard 1- and 2-categories is contained in the following result.

\begin{prop}\label{cor:_1adj_on_pgms_and_pics}
  The functors $\cD \mapsto \cD_1$ and $d$ induce adjunctions between
  \begin{itemize}
  \item the categories $\pgm$ and $\pcat$, and
  \item the category of strict Picard 2-categories and the category of
    strict Picard categories.
  \end{itemize}
  The counits of these adjunctions are both identities.
\end{prop}
\begin{proof}
  It is immediate from \cref{prop:_1ssm} and the definitions that
  applying $\cD \mapsto \cD_1$ to a permutative Gray-monoid gives a
  permutative category, and that the resulting permutative category is
  a strict Picard category if $\cD$ is a strict Picard 2-category;
  this constructs both left adjoints.  To construct the right
  adjoints, let $(C, \oplus, e)$ be a permutative category.  We must
  equip $dC$ with a permutative Gray-monoid structure.  The tensor
  product is given by
  \[
  d C \otimes d C \stackrel{\chi_{C,C}}{\longrightarrow} 
  d (C \times C) \stackrel{d \oplus}{\longrightarrow} 
  dC
  \]
  using \cref{prop:_1ssm} or the explicit description in
  \cref{rmk:d_*monoidal}.  The 2-natural isomorphism $\beta^{d C}$ is
  $d ( \beta^{C}) * \chi_{C,C}$, using the fact that $d(\tau^{\times})
  \circ \chi = \chi \circ \tau^{\otimes}$ by the second part of
  \cref{prop:_1ssm}.  The permutative Gray-monoid axioms for $dC$ then
  reduce to the permutative category axioms for $C$ and the lax
  symmetric monoidal functor axioms for $d$. Once again, $dC$ is a
  strict Picard 2-category if $C$ is a strict Picard category.  The
  statement about counits follows from the corresponding statement
  about the counit for the adjunction $\pi_0 \dashv d$, and the unit
  is a strict symmetric monoidal 2-functor by inspection.  The
  triangle identities then follow from those for $\pi_0 \dashv d$,
  concluding the construction of both adjunctions.
\end{proof}

\begin{rmk}
  The proof above is simple, but not entirely formal: while symmetric
  monoidal categories are the symmetric pseudomonoids in the symmetric
  monoidal 2-category $\cat$, permutative Gray-monoids do not admit
  such a description due to the poor interaction between the Gray
  tensor product and 2-natural transformations.
\end{rmk}

We now move on to our second adjunction between permutative categories
and permutative Gray-monoids which restricts to one between strict
Picard categories and strict Picard 2-categories.  This adjunction
models loop and suspension functors, and appears informally in work of
Baez and Dolan \cite{BD1995Higher} on stabilization phenomena in
higher categories.

\begin{lem}\label{lem:permcat_to_pgm}
  Let $(C, \oplus, e)$ be a permutative category with symmetry $\si$.
  Then the 2-category $\Si C$ with one object $*$, hom-category $\Si C
  (*,*) = C$, and horizontal composition given by $\oplus$ admits the
  structure of a permutative Gray-monoid $(\Si C, \wt{\oplus})$.
  The assignment $(C,\oplus) \mapsto (\Si C, \wt{\oplus})$ is the
  function on objects of a functor
  \[
  \Si \cn \pcat \to \pgm.
  \]
\end{lem}
\begin{proof}
  Since $C$ is a strict monoidal category, $\Si C$ is a strict
  2-category when horizontal composition is given by $\oplus$.  We can
  define a 2-functor $\wt{\oplus}\cn \Si C \otimes \Si C \to \Si C$ as
  the unique function on 0-cells, by sending any cell of the form $a
  \otimes 1$ to $a$, any cell of the form $1 \otimes b$ to $b$, and
  $\Si_{a,b}$ to the symmetry $\si_{a,b}\cn a \oplus b \cong b \oplus
  a$.  With the unique object as the unit, it is simple to check that
  this 2-functor makes $\Si C$ into a Gray-monoid.  All that remains
  is to define $\beta$ and check the three axioms.  Since there is
  only one object and it is the unit, the second axiom shows that the
  unique component of $\beta$ must be the identity 1-cell.  Then
  naturality on 1-cells is immediate, and the only two-dimensional
  naturality that is not obvious is for the cells $\Si_{a,b}$.  This
  axiom becomes the equation
  \[
  \be \oplus \Si_{a,b} = \Si^{-1}_{b,a} \oplus \be
  \]
  which is merely the claim that $\si_{a,b}$ is a symmetry rather than
  a braid.  It is then obvious that this assignment defines a functor
  as stated.
\end{proof}

\begin{example}
  The permutative Gray-monoid constructed in \cite[Example
  2.30]{Sch2011Classification} is a suspension $\Si C$ for the
  following permutative category $C$.
  \begin{itemize}
  \item The objects of $C$ are the elements of $\bZ/2$ with the 
    monoidal structure given by addition.
  \item Each endomorphism monoid of $C$ is $\bZ/2$ and there are no
    morphisms between distinct objects.
  \item The symmetry of the non-unit object with itself is the
    nontrivial morphism. 
  \end{itemize}
\end{example}
\begin{rmk}
  It is natural to expect that the permutative Gray-monoid $\Si C$ in
  the previous example models the 0-connected cover of the 2-type of
  the sphere spectrum, and indeed this will follow from
  \cref{thm:sigmak_equals_ksigma}.  One might also hope that a
  skeletal model for the sphere spectrum can be constructed as a
  ``many-object'' version of $\Si C$ together with an appropriate
  symmetry.  However \cref{thm:nostrskelmod} will prove that this is
  not possible.
\end{rmk}

\begin{lem}\label{lem:pgm_to_permcat}
  Let $(\cD,\oplus,e)$ be a permutative Gray-monoid.  Then the
  category $\cD(e,e)$ is a permutative category, with tensor product
  given by composition.  The assignment $\cD \mapsto \cD(e,e)$ is the
  function on objects of a functor
  \[
  \Om \cn \pgm \to \pcat.
  \]
\end{lem}
\begin{proof}
  For a Gray-monoid $\cD$, the hom-category $\cD(e,e)$ is a braided,
  strict monoidal category \cite{GPS95Coherence,CG2011periodicII} in
  which the tensor product is given by composition and the braid $f
  \circ g \cong g \circ f$ is the morphism $\Si_{f,g}$ in $\cD(e,e)$;
  we note that $fe= f$ and $eg = g$ since all the 1-cells involved are
  endomorphisms of the unit object, and the unit object in a
  Gray-monoid is a strict two-sided unit.  The component $\beta_{e,e}$
  is necessarily the identity, and the calculations in the proof of
  \cref{lem:permcat_to_pgm} show that $\Si_{f,g} = \Si_{g,f}^{-1}$, so
  we have a permutative structure on $\cD(e,e)$.
\end{proof}

\begin{prop}\label{prop:pgm_permcat_adj}
  The functor $\Si \cn \pcat \to \pgm$ is left adjoint to the functor
  $\Om \cn \pgm \to \pcat.$
\end{prop}
\begin{proof}
  It is easy to check that the composite $\Om \Si$ is the identity
  functor on $\pcat$, and we take this equality to be the unit of the
  adjunction.  The counit would be a functor $\Si \big( \cD(e,e) \big)
  \to \cD$ which we must define to send the single object of $\Si \big(
  \cD(e,e) \big)$ to the unit object $e$ of $\cD$ and then to be the
  obvious inclusion on the single hom-category.  This is clearly a
  2-functor, and the arguments in the proofs of the previous two
  lemmas show that this is a strict map of permutative Gray-monoids.

  The counit is then obviously the identity on the only hom-category
  when $\cD$ has a single object, and this statement is in fact the
  commutativity of one of the triangle identities for the adjunction.
  It is simple to check that $\Om$ applied to the counit is the
  identity as well since the counit is the identity functor when
  restricted to the hom-category of the unit objects, and this is the
  other triangle identity, completing the verification of the
  adjunction.
\end{proof}

Since the unit $1 \Rightarrow \Om \Si$ is the identity, and the counit
is an isomorphism on permutative Gray-monoids with one object, we have
the following corollary.

\begin{cor}\label{cor:equivs}
  The adjunction $\Si \dashv \Om$ in \cref{prop:pgm_permcat_adj}
  restricts to the categories of strict Picard categories and strict
  Picard 2-categories. Moreover, this adjunction gives equivalences
  between
  \begin{itemize}
  \item the category of permutative categories and the category of
    one-object permutative Gray-monoids, and
  \item the category of strict Picard categories and the category of
    one-object strict Picard 2-categories.
\end{itemize}
\end{cor}
\begin{proof}
  The first statement follows from the definitions, since both $\Si$
  and $\Om$ send strict Picard objects in one category to strict
  Picard objects in the other.  The other two statements are obvious
  from the proof above.
\end{proof}

\section{Stable homotopy theory of Picard 2-categories}
\label{sec:hty-thy-Pic-cats}
In this section we describe how to use the algebra of Picard
2-categories to express homotopical features of their corresponding
connective spectra categorically.  We begin with a brief review of
stable Postnikov towers, mainly for the purpose of fixing notation.
Subsequently, we identify algebraic models for this homotopical data
in terms of the categorical structure present in a Picard 2-category.

For an abelian group $\pi$, the Eilenberg-Mac
Lane spectrum of $\pi$ is denoted $H\pi$.  Its $n$th suspension is
denoted $\Si^{n}H\pi$, and has zeroth space given by the Eilenberg-Mac
Lane space $K(\pi,n)$.  With this notation, the stable Postnikov tower
of a connective spectrum $X$ is given as follows.
\[\begin{xy}
  0;<20mm,0mm>:<0mm,12mm>:: 
  (0,0)*+{X_0}="x0";
  (1,0)*+{\Si^{2}H(\pi_1X)}="h12";
  (-1,1)*+{\Si^{1}H(\pi_1X)}="h11";
  (0,1)*+{X_1}="x1";
  (1,1)*+{\Si^{3}H(\pi_2X)}="h23";
  (-1,2)*+{\Si^{2}H(\pi_2X)}="h22";
  (0,2)*+{X_2}="x2";
  (1,2)*+{\Si^{4}H(\pi_3X)}="h34";
  (0,3)*+{\vdots}="xn";
  {\ar^-{k_0} "x0"; "h12"};
  {\ar^-{k_1} "x1"; "h23"};
  {\ar^-{k_2} "x2"; "h34"};
  {\ar^-{i_1} "h11"; "x1"};
  {\ar^-{i_2} "h22"; "x2"};
  {\ar "xn"; "x2"};
  {\ar "x2"; "x1"};
  {\ar "x1"; "x0"};
\end{xy}\]

Since $X$ is connective, it follows that $X_0=H(\pi_0 X)$ and $k_0$ is therefore 
a stable map from $H(\pi_0 X)$ to $\Si^2 H(\pi_1 X)$. When $X$ is the $K$-theory
spectrum of a strict Picard $2$-category, we will model $k_0$ and
$k_1i_1$ algebraically via stable quadratic maps. A \emph{stable
  quadratic map} is a homomorphism from an abelian group $A$ to the
$2$-torsion of an abelian group $B$. The abelian group of stable
homotopy classes $[HA,\Si^2 HB]$ is naturally isomorphic to the
abelian group of stable quadratic maps $A\to B$ by
\cite[Equation (27.1)]{EM54groupsII}. Moreover \cite[Theorem 20.1]{EM54groupsIII}
implies that under this identification $k_0\colon H(\pi_0 X)\to \Si^2
H(\pi_1 X)$ corresponds to the stable quadratic map $\pi_0 X\to \pi_1
X$ given by precomposition with the Hopf map $\eta \colon\Si
\mathbb{S}\to \mathbb{S}$ where $\mathbb{S}$ denotes the sphere
spectrum.

The stable Postnikov tower can be constructed naturally in $X$, so
that if
\[
X' \to X
\]
is a map of spectra, we have the following commuting naturality
diagram of stable Postnikov layers.
\begin{equation}\label{eqn:postnikov-nat}
\begin{xy}
  0;<27mm,0mm>:<0mm,15mm>:: 
  (0,-.5)*+{\Si^{n}H(\pi_nX)}="k11";
  (1,-.5)*+{X_n}="x1";
  (2,-.5)*+{\Si^{n+2}H(\pi_{n+1}X)}="k23";
  (0,.5)*+{\Si^{n}H(\pi_nX')}="k11m";
  (1,.5)*+{X_n'}="x1m";
  (2,.5)*+{\Si^{n+2}H(\pi_{n+1}X')}="k23m";
  {\ar^-{i_n} "k11"; "x1"};
  {\ar^-{k_n} "x1"; "k23"};
  {\ar^-{i_n'} "k11m"; "x1m"};
  {\ar^-{k_n'} "x1m"; "k23m"};
  {\ar_-{} "k11m"; "k11"};
  {\ar "x1m"; "x1"};
  {\ar^-{} "k23m"; "k23"};
\end{xy}
\end{equation}

Picard 2-categories model stable $2$-types via $K$-theory.  The
$K$-theory functors for symmetric monoidal $n$-categories, constructed
in \cite{Seg74Categories,Tho95Symmetric,Man10Inverse} for $n = 1$ and
\cite{GJO2015KTheory} for $n = 2$, give faithful embeddings of Picard
$n$-categories into stable homotopy.  For the purposes of this section
we can take $K$-theory largely as a black box; in
\cref{sec:suspension-models-suspension} we give necessary definitions
and properties.

\subsection{Modeling stable Postnikov data}\label{sec:postnikov-data}
For a Picard category $(C,\oplus,e)$, the two possibly nontrivial
stable homotopy groups of its $K$-theory spectrum $K(C)$ are given by
\[
  \begin{array}{rcl}
    \pi_{0} K(C) & \cong &  \textrm{ob} C/\{x \sim y \textrm{ if there exists a 1-cell $f \cn x \to y$}\}\\
    \pi_{1} K(C) & \cong &  C(e,e).
  \end{array}
\]
The stable homotopy groups of the $K$-theory spectrum of a strict
Picard 2-category can be calculated similarly. We denote the
classifying space of a 2-category $\cD$ by $B\cD$ \cite{CCG10Nerves}.

\begin{lem}\label{lem:st_htpy_grps}
  Let $\cD$ be a strict Picard 2-category. The classifying space
  $B\cD$ is equivalent to $\Om^\infty K(\cD)$. The stable homotopy groups
  $\pi_{i} K(\cD)$ are zero except when $0 \leq i \leq 2$, in which
  case they are given by the formulas below.
  \[
  \begin{array}{rcl}
    \pi_{0} K(\cD) & \cong &  \textrm{ob}\cD/\{x \sim y \textrm{ if there exists a 1-cell $f \cn x \to y$}\}\\
    \pi_{1} K(\cD) & \cong &  \textrm{ob}\cD(e,e)/\{f \sim g \textrm{ if there exists a 2-cell $\al \cn f \Rightarrow g$}\}\\
    \pi_{2} K(\cD) & \cong & \cD(e,e)(\id_{e}, \id_{e})
  \end{array}
  \]
\end{lem}
\begin{proof}
  First, note that $\cD$ has underlying 2-category a bigroupoid, and
  the above are the unstable homotopy groups of the pointed space
  $(B\cD, e)$ by \cite[Remark 4.4]{CCG10Nerves}.
  Since the objects of $\cD$ are invertible, the space $B\cD$ is group-complete, and  hence it is the zeroth space of the $\Om$-spectrum $K(\cD)$. Thus the stable homotopy groups of $K(\cD)$ agree with
  the unstable ones for $B\cD$.
\end{proof}

\begin{prop}\label{prop:st-equiv-cat-equiv}
  A map of strict Picard 2-categories induces a stable equivalence of
  $K$-theory spectra if and only if it is an equivalence of Picard
  2-categories.
\end{prop}
\begin{proof}
  Note that the existence of inverses in a Picard 2-category implies
  that for any object $x$ we have an equivalence of categories
  $\cD(e,e)\simeq \cD(x,x)$ induced by translation by $x$. Similarly,
  for any 1-morphism $f\colon e \to e$ there is an isomorphism of sets
  $\cD(e,e)(\id_e,\id_e) \cong \cD(e,e)(f,f)$ induced by translation
  by $f$.

  A map $F\colon \cD \to \cD'$ of strict Picard 2-categories is a
  categorical equivalence if and only if it is an equivalence of
  underlying 2-categories, that is, if it is biessentially surjective
  and a local equivalence (see \cite[Section 5]{Gur2012Biequivalences} and
  \cite[Theorem 2.25]{Sch2011Classification}). By
  \cref{lem:st_htpy_grps} and the observation above, this happens
  exactly when $f$ induces an isomorphism on the stable homotopy
  groups of the corresponding $K$-theory spectra.
\end{proof}

We will use the adjunctions from \cref{sec:two-adjunctions} to reduce
the calculation of the stable quadratic maps corresponding to $k_0$
and $k_1i_1$ of $K(\cD)$ to two instances of the calculation of $k_0$
in the $1$-dimensional case.

\begin{lem}[\cite{JO12Modeling}]\label{lem:k0-given-by-symm-1}
  Let $C$ be a strict Picard category with unit $e$ and symmetry
  $\beta$.  Then the bottom stable Postnikov invariant $k_0\colon
  H\pi_0K(C)\to \Si^2 H\pi_1K(C)$ is modeled by the stable quadratic
  map $k_0\colon\pi_0K(C)\to \pi_1K(C)$,
  \[
  [x]\mapsto (e \fto{\cong} x\,x\,x^*\,x^* \fto{\beta_{x,x}\,x^*\,x^*} x\,x\,x^*\,x^* \fto{\cong} e),
  \]
  where $x$ is an object in $C$ and $x^*$ denotes an inverse of $x$.
\end{lem}
\begin{rmk}
  The middle term of the composite $k_0(x)$ was studied in
  \cite{sinh1975,JS1993btc} and is called the \emph{signature} of $x$.
\end{rmk}
\begin{proof}[Proof of \cref{lem:k0-given-by-symm-1}]
  Note that $k_0\colon \pi_0 K(C)\to \pi_1 K(C)$ is a well-defined
  function (does not depend on the choices of $x$, $x^*$, and
  $xx^*\cong e$). Indeed, given isomorphisms $x\cong y$, $xx^*\cong e$
  and $yy^*\cong e$, there is a unique isomorphism $j\colon x^*\cong
  y^*$ such that
  \[
  \xymatrix{yx^*\ar[r]\ar[d]_{yj} & xx^*\ar[d] \\
  yy^* \ar[r] & e}
  \]
  commutes.

  Moreover, it is clear that $k_0$ is compatible with equivalences of
  Picard categories. By \cite[Theorem~2.2]{JO12Modeling}, we can thus
  replace $C$ by a strict skeletal Picard category. In
  [\textit{loc.~cit.}, Section~3], a natural action $S\times C\to C$
  is defined, where $S$ is a strict skeletal model for the
  1-truncation of the sphere spectrum. It follows from the definition
  of the action that
  \[
  \pi_1(BS)\times \pi_1(BC,x)\to \pi_1(BC,e)
  \] 
  sends $(\eta,\id_x)$ to $\beta_{x,x}x^*x^*$, where $\eta$ denotes
  the generator of $\pi_1(BS)\cong \bZ/2$. Finally it follows from
  [\textit{loc.~cit.}, Proposition~3.4] that the action $S \times C
  \to C$ models the truncation of the action of the sphere spectrum on
  $KC$, thus the image under the action of $(\eta,\id_x)$ agrees with
  the image of $[x]$ under the stable quadratic map associated to the
  bottom stable Postnikov invariant.
\end{proof}

\begin{prop}\label{lem:1-truncation}
  Let $\cD$ be a strict Picard 2-category and let $\cD\to d(\cD_1)$ be
  the unit of the adjunction in
  \cref{cor:_1adj_on_pgms_and_pics}. Then
  \[
  K(\cD)\to K\big(d(\cD_1)\big)
  \]
  is the 1-truncation of $K(\cD)$.
\end{prop}
\begin{proof}
  Using the formulas in \cref{lem:st_htpy_grps}, it is clear that
  $\cD\to d(\cD_1)$ induces an isomorphism on $\pi_0$ and $\pi_1$, and
  that $\pi_2 K\big( d(\cD_1)\big)=0$. Moreover, both $K$-theory
  spectra have $\pi_i = 0$ for $i > 2$, so $\cD_1$ models the
  1-truncation of $\cD$.
\end{proof}

\begin{lem}\label{lem:k-equals-k}
  For any permutative category $C$, the $K$-theory spectrum of $C$ is
  stably equivalent to the $K$-theory spectrum of the corresponding
  permutative Gray-monoid, $dC$.
\end{lem}
\begin{proof}
  This follows directly from the formulas in \cite{GJO2015KTheory},
  and in particular Remark 6.32.
\end{proof}

For any connective spectrum $X$, the bottom stable Postnikov invariant
of $X$ and its 1-truncation $X_1$ agree. Thus combining
\cref{lem:k0-given-by-symm-1}, \cref{lem:1-truncation} and
\cref{lem:k-equals-k} yields the following result.
\begin{cor}\label{lem:k0-given-by-symm-2}
  Let $\cD$ be a strict Picard 2-category with unit $e$ and symmetry
  $\beta$.  Then the bottom stable Postnikov invariant $k_0\colon
  H\pi_0K(\cD)\to \Si^2 H\pi_1K(\cD)$ is modeled by the stable
  quadratic map $k_0\colon\pi_0K(\cD)\to \pi_1 K(\cD)$,
  \[
  [x]\mapsto [e \fto{\hty} x\,x\,x^*\,x^* \fto{\beta_{x,x}\,x^*\,x^*} x\,x\,x^*\,x^* \fto{\hty} e],
  \]
  where $x$ is an object in $\cD$ and $x^*$ denotes an inverse of $x$.
\end{cor}
\begin{rmk}\label{rmk:sign_welldefined}
  It can be checked directly that the function
  $k_0\colon\textrm{ob}(\cD)\to \pi_1(K\cD)$ is well-defined using the
  essential uniqueness of the inverse: given another object
  $\overline{x}$ together with an equivalence $e \hty x \overline{x}$,
  there is an equivalence $x^* \hty \overline{x}$ and an isomorphism
  2-cell in the obvious triangle which is unique up to unique
  isomorphism.  This follows from the techniques in
  \cite{Gur2012Biequivalences}, and many of the details are explained
  there in Section 6.
\end{rmk}

In order to identify the composite $k_1i_1$ categorically, we analyze
the relationship between Postnikov layers and categorical suspension.

\begin{prop}\label{prop:susp-conn-cover}
  Let $\cD$ be a strict Picard 2-category and let $\Si\Omega\cD \to
  \cD$ be the counit of the adjunction in \cref{prop:pgm_permcat_adj}.
  Then
  \[
  K \big(\Si\Omega\cD\big) \to K (\cD)
  \]
  is a 0-connected cover of $K (\cD)$.
\end{prop}
\begin{proof}
  It is clear from the formulas in \cref{lem:st_htpy_grps} that
  $\Si\Omega\cD \to \cD$ induces an isomorphism on $\pi_1$ and
  $\pi_2$, and moreover, the corresponding $K$-theory spectra have
  $\pi_i = 0$ for $i > 2$. Since $\Si\Omega\cD$ has only one object,
  we have $\pi_0 K(\Si \Omega \cD)=0$, so $\Si\Omega\cD$ models the
  0-connected cover of $\cD$.
\end{proof}

In addition to the elementary algebra and homotopy theory of Picard
2-categories discussed above, we require the following result.
\begin{thm}\label{thm:sigmak_equals_ksigma}
  Let $C$ be a permutative category.  Then $\Si K(C)$ and $K(\Si C)$
  are stably equivalent.
\end{thm}
The proof of \cref{thm:sigmak_equals_ksigma} requires a nontrivial
application of 2-monad theory.  We develop the relevant 2-monadic
techniques in \cref{sec:strictification} and give the proof in
\cref{sec:suspension-models-suspension}.  These two sections are
independent of the preceding sections.

\begin{lem}\label{lem:k1i1-given-by-symm-2}
  Let $(\cD, \oplus, e)$ be a strict Picard 2-category. Then the
  composite
  \[
  k_1 i_1\colon \Si H\pi_1 K(\cD)\to \Si^3 H\pi_2 K(\cD)
  \]
  is modeled by the stable quadratic map $\pi_1 K(\cD) \to \pi_2
  K(\cD)$, 
  \[
  \quad [f]\mapsto (\id_e \fto{\iso} 
  f\circ f\circ f^*\circ f^* \fto{\Si_{f,f}\,f^*\circ f^*} 
  f\circ f\circ f^*\circ f^* \fto{\iso} 
  \id_e), 
  \] 
  where $f\colon e\to e$ is a 1-cell in $\cD$ and $f^*$ denotes an
  inverse of $f$.
\end{lem}
\begin{proof}
  We use superscripts to distinguish Postnikov data of different
  spectra. The composite $k_1^\cD i_1^\cD$ in the first Postnikov
  layer of the spectrum $K(\cD)$ identifies with the composite $k_1
  ^{\Si\Omega\cD} i_1^{\Si \Omega\cD}$ since $K(\Si\Omega\cD)$
  is the $0$-connected cover of $K(\cD)$ by
  \cref{prop:susp-conn-cover} and the Postnikov tower can be
  constructed naturally (\cref{eqn:postnikov-nat}).
  
  Since $K(\Si\Omega\cD)\hty \Si K(\Omega \cD)$ by
  \cref{thm:sigmak_equals_ksigma} and $K(\Omega\cD)$ is connective,
  it follows that
  \[
  k_1 ^{\Si\Omega\cD} i_1^{\Si \Omega\cD} = \Si(k_0^{\Omega\cD}i_0^{\Omega\cD})=\Si(k_0^{\Omega \cD})
  \]
  in the stable homotopy category.

  Finally, we deduce from \cref{lem:k0-given-by-symm-1} that the map
  $\Si(k_0^{\Omega\cD})$ is represented by the desired group
  homomorphism.
\end{proof}  

\subsection{Application to strict skeletal Picard 2-categories}
Now we make an observation about the structure 2-cells $\Sigma_{f,g}$
in a strict Picard 2-category.  This algebra will be a key input for
our main application,
\cref{thm:nostrskelmod}.
\begin{lem}\label{lem:sigma-triv}
  Let $(\cD, \oplus, e)$ be a strict Picard 2-category.  Let $g\cn e \to
  e$ be any 1-cell and let $s = \beta_{x,x}\,x^*\,x^*$ be a
  representative of the signature of some object $x$ with inverse
  $x^*$.  Then $\Sigma_{s,g}$ and $\Si_{g,s}$ are identity 2-cells in $\cD$.
\end{lem}
\begin{proof}
  By naturality of the symmetry and interchange,
  $\Sigma_{\beta_{y,z},h}$ and $\Sigma_{h,\be_{y,z}}$ are identity
  2-cells for any 1-cell $h$ \cite[Proposition 3.41]{GJO2015KTheory}.  The
  result for $\Si_{g,s}$ follows by noting that $\Si_{g,f}w = \Si_{g,
    fw}$ for any 1-cells $f,g$ and object $w$ by the associativity
  axiom for a Gray-monoid.  Hence
  $\Si_{g,s}=\Si_{g,\be_{x,x}x^*x^*}=\Si_{g,\be_{x,x}}x^*x^*$, which
  is the identity 2-cell.
    
  For the other equality, we note the final axiom of
  \cite[Proposition 3.3]{Gurski13Coherence}
  reduces to the following equality of pasting diagrams for objects
  $y$, $z$, $w$ with endomorphisms $t_y$, $t_z$, $t_w$ respectively.
  \[
  \def\objectstyle{\scriptstyle}
  \def\labelstyle{\textstyle}
  \begin{array}{c}
  \begin{xy}
    0;<.85mm,0mm>:<0mm,1mm>:: 
    (0,0)*+{yzw}="LL";
    (20,20)*+{yzw}="TL";
    (20,-20)*+{yzw}="BL";
    (50,20)*+{yzw}="TR";
    (50,-20)*+{yzw}="BR";
    (70,0)*+{yzw}="RR";
    (30,0)*+{yzw}="CC";  
    {\ar^-{\scriptstyle t_y \, z \, w} "LL"; "TL"};
    {\ar^-{\scriptstyle y \, t_z \, w} "TL"; "TR"};
    {\ar^-{\scriptstyle y \, z \, t_w} "TR"; "RR"};
    {\ar_-{\scriptstyle y \, z \, t_w} "LL"; "BL"};
    {\ar_-{\scriptstyle y \, t_z \, w} "BL"; "BR"};
    {\ar_-{\scriptstyle t_y \, z \, w} "BR"; "RR"};
    {\ar^-{\scriptstyle y \, t_z \, w} "LL"; "CC"};
    {\ar_-{\scriptstyle t_y \, z \, w} "CC"; "TR"};
    {\ar^-{\scriptstyle y \, z \, t_w} "CC"; "BR"};
    {\ar@{=>}_-{\Sigma_{t_y,t_z} \, w} "CC"+(-5,8); "TL"+(3,-7) };
    {\ar@{=>}^-{y \, \Sigma_{t_z,t_w}} "BL"+(5,8); "CC"+(-2,-7) };
    {\ar@{=>}_-{\ \Sigma_{(t_y \, z), t_w}} "BR"+(-3,18); "TR"+(-3,-18) };
  \end{xy}
  \begin{xy}
    0;<2.2mm,0mm>:<0mm,2.2mm>:: 
    {\ar@{=}^{} (1,0)*+{}; (-1,0)*+{};};
  \end{xy}
  \begin{xy}
    0;<.85mm,0mm>:<0mm,1mm>:: 
    (0,0)*+{yzw}="LL";
    (20,20)*+{yzw}="TL";
    (20,-20)*+{yzw}="BL";
    (50,20)*+{yzw}="TR";
    (50,-20)*+{yzw}="BR";
    (70,0)*+{yzw}="RR";
    (40,0)*+{yzw}="CC";  
    {\ar^-{\scriptstyle t_y \, z \, w} "LL"; "TL"};
    {\ar^-{\scriptstyle y \, t_z \, w} "TL"; "TR"};
    {\ar^-{\scriptstyle y \, z \, t_w} "TR"; "RR"};
    {\ar_-{\scriptstyle y \, z \, t_w} "LL"; "BL"};
    {\ar_-{\scriptstyle y \, t_z \, w} "BL"; "BR"};
    {\ar_-{\scriptstyle t_y \, z \, w} "BR"; "RR"};
    {\ar_-{\scriptstyle y \, z \, t_w} "TL"; "CC"};
    {\ar^-{\scriptstyle t_y \, z \, w} "BL"; "CC"};
    {\ar^-{\scriptstyle y \, t_z \, w} "CC"; "RR"};
    {\ar@{=>}^-{\Sigma_{t_y, (z \, t_w)}} "BL"+(3,18); "TL"+(3,-18) };
    {\ar@{=>}_-{\Sigma_{t_y,t_z} \, w} "BR"+(-5,8); "CC"+(3,-7) };
    {\ar@{=>}^-{y \, \Sigma_{t_z,t_w}} "CC"+(5,8); "TR"+(-2,-7) };
  \end{xy}
  \end{array}
  \]
  Thus the result for $\Si_{s,g}$ follows by taking $(y, z, w) =
  (x\,x, x^*\,x^*, e)$, $t_y = \beta_{x,x}$, $t_z = \id$, and $t_w =
  g$.
\end{proof}

We are now ready to give our main application
regarding stable Postnikov data of strict skeletal Picard 2-categories.
\begin{thm}\label{thm:nostrskelmod}
  Let $\cD$ be a strict skeletal Picard 2-category and assume that
  \[k_0\colon \pi_0 K(\cD) \to \pi_1 K(\cD)\]
  is surjective. Then $k_1i_1$ is trivial.
\end{thm}

\begin{proof}
  We prove that the stable quadratic map $\pi_1 K(\cD) \to \pi_2
  K(\cD)$ from \cref{lem:k1i1-given-by-symm-2} that models the
  composite $k_1i_1$ is trivial.  Since $k_0$ is surjective by
  assumption, it suffices to consider $k_1i_1(f)$ for $f$ of the form
  \begin{equation}\label{eqn:k0_decomp}
  e \fto{w} x\,x\,x^*\,x^* \fto{\beta_{x,x}\,x^*\,x^*}
  x\,x\,x^*\,x^* \fto{w^*} e  
  \end{equation}
  for some object $x$ with inverse $x^*$. Here $w$ denotes the
  composite 
  \[
  e \fto{u} x\,x^* \fto{x\,u\,x^*} x\,x\,x^*\,x^*
  \]
  for a chosen equivalence $u\cn e\hty x\,x^*$ and $w^*$ denotes the
  corresponding reverse composite for a chosen $u^*\cn x\,x^* \hty e$
  inverse to $u$.  Note that the isomorphism class of $f$ is
  independent of the choices of the inverse object $x^*$ and the
  equivalences $u$ and $u^*$ (see \cref{rmk:sign_welldefined}).  Since
  $\cD$ is skeletal, it must be that $xx^* = e$.  Therefore we can
  choose the equivalence $u\cn e\hty xx^*$ to be $\id_e$ and then
  choose $u^*$ to be $\id_e$ as well.  With these choices, the
  composite $f$ is actually equal to $\beta_{x,x}\,x^*\,x^*$. By
  \cref{lem:sigma-triv} the Gray structure 2-cell $\Sigma_{f,f}$ is
  the identity 2-cell $\id_{f \circ f}$. This implies that $k_1 i_1
  (f) = \id_{\id_e}$.
\end{proof}

\begin{rmk}\label{rmk:differential}
  The result of \cref{thm:nostrskelmod} may be viewed as the
  computation of a differential in the spectral sequence arising from
  mapping into the stable Postnikov tower of $K\cD$.  This spectral
  sequence appears, for example, in \cite{Kah66spectral} and is a
  cocellular construction of the Atiyah-Hirzebruch spectral sequence
  (see \cite[Appendix B]{GM95Generalized}).
\end{rmk}

Our most important application concerns the sphere spectrum. 
\begin{cor}\label{cor:no_model}
  Let $\cD$ be a strict skeletal Picard 2-category.  Then $\cD$ cannot
  be a model for the 2-truncation of the sphere spectrum.
\end{cor}
\begin{proof}
  The nontrivial element in $\pi_1$ of the sphere spectrum is given by
  $k_0(1)$, so $k_0$ is surjective and therefore
  \cref{thm:nostrskelmod} applies.  But $k_1i_1$ is $Sq^2$, which is the
  nontrivial element of $H^2(\bZ/2; \bZ/2)$ \cite[pp.~117--118]{MT1968Cohomology}.    
\end{proof}

\begin{rmk}
  To understand the meaning of this result, recall that one can
  specify a unique Picard category by choosing two abelian groups for
  $\pi_0$ and $\pi_1$ together with a stable quadratic map $k_0$ for the
  symmetry.  This is the content of \cref{thm:strskel_pic}.  However,
  one does not specify a Picard 2-category by simply choosing three
  abelian groups and two group homomorphisms.  This is tantamount to
  specifying a stable 2-type by choosing the bottom Postnikov
  invariant $k_0$ and the composite $k_1 i_1$.
  \Cref{thm:nostrskelmod} shows that such data do not always
  assemble to form a strict Picard 2-category.  For example, the construction
  of \cite[5.2]{Bar14Quasistrict} does not satisfy the axioms of a
  permutative Gray-monoid.
\end{rmk}

\section{Strictification via 2-monads}
\label{sec:strictification}

In this section we develop the 2-monadic tools used in the
proof of \cref{thm:sigmak_equals_ksigma}.  In
\cref{sec:review-2-monad-theory} we recall some basic definitions as well
 as abstract coherence theory from the perspective of 2-monads.  Our focus
  is on various strictification results for algebras and pseudoalgebras 
  over 2-monads, and how strictification can often be expressed as a
   2-adjunction with good properties.  In
\cref{sec:2_apps} we apply this to construct a strictification of
pseudodiagrams as a left 2-adjoint.  The material in
this section is largely standard 2-category theory, but we did not
know a single reference which collected it all in one place.

The formalism of this section aids the proof of
\cref{thm:sigmak_equals_ksigma} in two ways.  
First, it allows us to
produce strict diagrams of 2-categories by
working with diagrams which are weaker (e.g., whose arrows take values in
pseudofunctors) but more straightforward to define.  This occurs in
\cref{sec:constructions-of-K-theory}.  Second, it allows us to
construct strict equivalences of strict diagrams by working instead with pseudonatural
equivalences between them.  This occurs in
\cref{sec:proof-of-sigma-k-eq-k-sigma}.

\subsection{Review of 2-monad theory}\label{sec:review-2-monad-theory}
We recall relevant aspects of 2-monad theory and fix notation.  These
include maps of monads and abstract coherence theory
\cite{KS74Review,power1989coherence,BKP1989Two,Lac02Codescent}.  Let $\cA$ be a
2-category, and $(T \cn \cA \to \cA, \eta, \mu)$ be a 2-monad on
$\cA$.  We then have the following 2-categories of algebras and
morphisms with varying levels of strictness.
\begin{enumerate}
\item $\talgs$ is the 2-category of strict $T$-algebras, strict
  morphisms, and algebra 2-cells.  Its underlying category is just the
  usual category of algebras for the underlying monad of $T$ on the
  underlying category of $\cA$.
\item $\talg$ is the 2-category of strict $T$-algebras,
  pseudo-$T$-morphisms, and algebra 2-cells.
\item $\pstalg$ is the 2-category of pseudo-$T$-algebras,
  pseudo-$T$-morphisms, and algebra 2-cells.
\end{enumerate}
We have inclusions and forgetful functors as below.
\[
\begin{xy}
  0;<25mm,0mm>:<0mm,12mm>:: 
  (0,1)*+{\talgs}="tas";
  (1,1)*+{\talg}="ta";
  (2,1)*+{\pstalg}="pta";
  (1,0)*+{\cA}="a";
  {\ar^-{i} "tas"; "ta"};
  {\ar "ta"; "pta"};
  {\ar_-{U} "tas"; "a"};
  {\ar_-{U} "ta"; "a"};
  {\ar^-{U} "pta"; "a"};
\end{xy}
\]
A map of 2-monads is precisely the data necessary to provide a
2-functor between 2-categories of strict algebras.
\begin{defn}\label{defn:map-of-2-monads}
  Let $S$ be a 2-monad on $\cA$ and $T$ a 2-monad on $\cB$.  A
  \emph{strict map of 2-monads} $S \to T$ consists of a 2-functor
  $F\cn \cA \to \cB$ and a 2-natural transformation $\la \cn TF
  \Rightarrow FS$ satisfying two compatibility axioms \cite{Bec1969Distributive}:
  \begin{align*}
    \la \circ \mu F & = F \mu \circ \la S \circ T \la\\
    \la \circ \eta F & = F \eta.
  \end{align*}
\end{defn}
\begin{prop}\label{prop:map=lift}
  If $F\cn S \to T$ is a strict map of 2-monads, then $F$ lifts to the
  indicated 2-functors in the following diagram.
  \[
  \begin{xy}
    0;<25mm,0mm>:<0mm,12mm>:: 
    (0,0)*+{\cA}="a0";
    (0,1)*+{S\mh\Alg}="a1";
    (0,2)*+{S\mh\Alg_s}="a2";
    (1,0)*+{\cB}="b0";
    (1,1)*+{T\mh\Alg}="b1";
    (1,2)*+{T\mh\Alg_s}="b2";
    {\ar_-{F} "a0"; "b0"};
    {\ar^-{F} "a1"; "b1"};
    {\ar^-{F} "a2"; "b2"};
    {\ar_-{i} "a2"; "a1"};
    {\ar_-{U} "a1"; "a0"};
    {\ar^-{i} "b2"; "b1"};
    {\ar^-{U} "b1"; "b0"};
  \end{xy}
  \]
\end{prop}

Abstract coherence theory provides left 2-adjoints to $\talgs
\hookrightarrow \talg$ and the composite $\talgs \hookrightarrow
\pstalg$.  Lack discusses possible hypotheses in \cite[Section
3]{Lac02Codescent}, so we give the following theorem in outline form.

\begin{thm}{{\cite[Section 3]{Lac02Codescent}}}\label{thm:2monad_str}
Under some assumptions on $\cA$ and $T$, the inclusions 
\[
i\cn\talgs \hookrightarrow \talg, \quad j\cn\talgs \hookrightarrow \pstalg
\]
have left 2-adjoints generically denoted $Q$.  Under even further
assumptions, the units $ 1 \Rightarrow iQ, 1 \Rightarrow jQ$ and the
counits $Qi \Rightarrow 1, Qj \Rightarrow 1$ of these 2-adjunctions have components which are internal equivalences in $\talg$ for $Q
\dashv i$ and $\pstalg$ for $Q \dashv j$, respectively. 
\end{thm}
\begin{rmk}\label{rmk:counits}
  The proofs in \cite{Lac02Codescent} only concern the units, but the
  statement about counits follows immediately from the 2-out-of-3
  property for equivalences and one of the triangle identities.  We 
  note that the components of the counits are actually always 1-cells 
  in $\talgs$, so saying they are equivalences in $\talg$ or $\pstalg$ 
  requires implicitly applying $i$ or $j$, respectively.
\end{rmk}

\begin{notn}\label{notn:i-and-j}
  We will always denote inclusions of the form $\talgs \hookrightarrow
  \talg$ by $i$, and inclusions of the form $\talgs \hookrightarrow
  \pstalg$ by $j$.  If we need to distinguish between the left
  adjoints for $i$ and $j$, we will denote them $Q_i$ and $Q_j$,
  respectively.
\end{notn}

\subsection{Two applications of 2-monads}\label{sec:2_apps}

We are interested in two applications of \cref{thm:2monad_str}: one
which gives 2-categories as the strict algebras
(\cref{prop:st-left-adj}), and one which gives 2-functors with fixed
domain and codomain as the strict algebras
(\cref{prop:functor_2monadic}).  Combining these in \cref{cor:J_Q_adj}
we obtain the main strictification result used in our analysis of $K$-theory 
and suspension in \cref{sec:suspension-models-suspension}.

We begin with the 2-monad for 2-categories and refer the interested
reader to \cite{Lac10Icons} and \cite{LP20082nerves} for further
details.

\begin{defn}\label{defn:graphs}
\ 
\begin{enumerate}
\item A \emph{category-enriched graph} or \emph{$\Cat$-graph} $\big( S,
  S(x,y) \big)$ consists of a set of objects $S$ and for each pair of
  objects $x,y \in S$, a category $S(x,y)$.
\item A \emph{map of $\Cat$-graphs} $(F,F_{x,y})\cn \big( S, S(x,y) \big)
  \to \big( T, T(w,z) \big)$ consists of a function $F\cn S \to T$ and
  a functor $F_{x,y} \cn S(x,y) \to T(Fx, Fy)$ for each pair of
  objects $x,y \in S$.
\item A \emph{$\Cat$-graph 2-cell} $\al \cn (F,F_{x,y}) \Rightarrow
  (G,G_{x,y})$ only exists when $F=G$ as functions $S \to T$, and then
  consists of a natural transformation $\al_{x,y} \cn F_{x,y}
  \Rightarrow G_{x,y}$ for each pair of objects $x,y \in S$.
\end{enumerate}
\end{defn}

\begin{notn}\label{defn:catgrph}
  $\Cat$-graphs, their maps, and 2-cells form a 2-category, $\catgrph$,
  with the obvious composition and unit structures.
\end{notn}

\begin{defn}\label{defn:icon}
  Let $\cA, \cB$ be 2-categories, and $F,G \cn \cA \to \cB$ be a pair
  of 2-functors between them.  An \emph{icon} $\al \cn F \Rightarrow
  G$ exists only when $Fa = Ga$ for all objects $a \in \cA$, and then
  consists of natural transformations
  \[
  \alpha_{a,b}:F_{a,b} \Rightarrow G_{a,b}:\cA(a,b) \rightarrow \cB(Fa,Fb)
  \]
  for all pairs of objects $a$, $b$, such that the following diagrams
  commute.  (Note that we suppress the 0-cell source and target
  subscripts for components of the transformations $\alpha_{a,b}$ and
  instead only list the 1-cell for which a given 2-cell is the
  component.)
  \[
  \xy
  {\ar^{=} (0,0)*+{\id_{Fa}}; (25,0)*+{F\id_{a}} };
  {\ar^{\alpha_{\id}} (25,0)*+{F\id_{a}}; (25,-15)*+{G\id_{a}} };
  {\ar_{=} (0,0)*+{\id_{Fa}}; (25,-15)*+{G\id_{a}} };
  {\ar^{=} (50,0)*+{Ff\circ Fg}; (80,0)*+{F(f\circ g)} };
  {\ar^{\alpha_{f\circ g}} (80,0)*+{F(f\circ g)}; (80,-15)*+{G(f\circ g)} };
  {\ar_{\alpha_{f}*\alpha_{g}} (50,0)*+{Ff\circ Fg}; (50,-15)*+{Gf\circ Gg} };
  {\ar_{=} (50,-15)*+{Gf\circ Gg}; (80,-15)*+{G(f\circ g)} };
  \endxy
  \]
\end{defn}

\begin{rmk}
  We can define icons between pseudofunctors or lax functors with only
  minor modifications, replacing some equalities above with the
  appropriate coherence cell; see \cite{LP20082nerves,Lac10Icons}.
\end{rmk}

\begin{notn}\label{notn:2cat2i}
  2-categories, 2-functors, and icons form a 2-category which we
  denote $\twocattwoi$.  2-categories, pseudofunctors, and icons form
  a 2-category which we denote $\twocatpi$. Bicategories,
  pseudofunctors, and icons also form a 2-category which we denote
  $\bicat_{\text{p,i}}$.
\end{notn}

Recall that a 2-functor $U \cn \cA \to \cK$ is 2-monadic if it has a
left 2-adjoint $F$ and $\cA$ is 2-equivalent to the 2-category of
algebras $(UF)\mh\mb{Alg}_{s}$ via the canonical comparison map.
\begin{prop}[{\cite{LP20082nerves,Lac10Icons}}]\label{prop:2cat2i_2monadic}
  The 2-functor $\twocattwoi \to \catgrph$ is 2-monadic, and
  the left 2-adjoint is given by the $\cat$-enriched version of the
  free category functor.
\end{prop}
The following is our first application of \cref{thm:2monad_str}.
\begin{prop}\label{prop:st-left-adj}
  The two inclusions, 
  \[
  i \cn \twocattwoi \hookrightarrow \twocatpi, \quad
  j \cn \twocattwoi \hookrightarrow \bicat_{\text{p,i}} 
  \]
  have left 2-adjoints, and the components of the units and counits of
  both adjunctions are internal equivalences in $\twocatpi$ for $Q_i
  \dashv i$ and $\bicat_{\text{p,i}}$ for $Q_j \dashv j$,
  respectively.
\end{prop}
\begin{proof}
  The induced monad $T$ on $\catgrph$ satisfies a version of the
  hypotheses for \cref{thm:2monad_str} (for example, it is a finitary
  monad) so we get left 2-adjoints to both inclusions
  \[
  i \cn \talgs \to \talg, \quad j \cn \talgs \to \pstalg.
  \]
  Now $\talg$ can be identified with $\twocatpi$, and one can check
  that $\pstalg$ can be identified with $\bicat_{\text{p,i}}$, and
  using these the two left 2-adjoints above are both given by the
  standard functorial strictification functor, often denoted
  $\textrm{st}$ (see \cite{JS1993btc} for the version with only a
  single object, i.e., monoidal categories).  The objects of
  $\textrm{st}(X)$ are the same as $X$, while the 1-cells are formal
  strings of composable 1-cells (including the empty string at each
  object).  Internal equivalences in either $\talg$ or $\pstalg$ for
  the 2-monad $T$ are bijective-on-objects biequivalences, and it is
  easy to check that the unit is such; see
  \cite{LP20082nerves,Gur2013monoidal} for further details.
\end{proof}

\begin{rmk}\label{rmk:2cat2i_cocomp}
  We should note that $\twocattwoi$ is complete and cocomplete as a 2-category,
  since it is the 2-category of algebras for a finitary 2-monad on a complete and 
  cocomplete 2-category.  This will be necessary for later
  constructions.  On the other hand, $\twocatpi$ is not cocomplete as
  a 2-category, but is as a bicategory: coequalizers of pseudofunctors
  rarely exist in the strict, 2-categorical sense, but all
  bicategorical colimits do exist.
\end{rmk}

Our second application of \cref{thm:2monad_str} deals with functor
2-categories.  Here we fix a small 2-category $\cA$ and a complete and
cocomplete 2-category $\cK$.

\begin{notn}\label{notn:functor2cats}
  Let $[\cA, \cK]$ denote the 2-category of 2-functors, 2-natural
  transformations, and modifications from $\cA$ to $\cK$.  Let
  $\bicat(\cA, \cK)$ denote the 2-category of pseudofunctors,
  pseudonatural transformations, and modifications from $\cA$ to
  $\cK$.  Let $\gray(\cA, \cK)$ denote the 2-category of 2-functors,
  pseudonatural transformations, and modifications from $\cA$ to
  $\cK$.  This is the internal hom-object corresponding to the Gray
  tensor product on $\iicat$ \cite{GPS95Coherence}.
\end{notn}

\begin{rmk}
$\bicat(\cA, \cK)$ inherits its compositional and unit structures from 
the target 2-category $\cK$ and is therefore a 2-category rather than a
bicategory even though all of its cells are of the weaker, bicategorical variety.
\end{rmk}

Let $\textrm{ob} \, \cA$ denote the discrete 2-category with the
same set of objects as $\cA$.  We have an inclusion $\textrm{ob} \,
\cA \hookrightarrow \cA$ which induces a 2-functor $U \cn [\cA,
\cK] \to [\textrm{ob} \, \cA, \cK]$.  
\begin{prop}\label{prop:functor_2monadic}
  The forgetful 2-functor $U\cn[\cA, \cK] \to [\ob\, \cA, \cK]$ is
  2-monadic, and the left 2-adjoint is given by enriched left Kan
  extension.  The induced 2-monad preserves all colimits, and so the
  inclusions
 \[
  i \cn [\cA, \cK]  \hookrightarrow  \gray (\cA, \cK), \quad
  j \cn [\cA, \cK]  \hookrightarrow  \bicat(\cA, \cK)
 \] 
 have left 2-adjoints.  The units and counits of these adjunctions
 have components which are internal equivalences in $\gray(\cA, \cK)$
 for $Q_i \dashv i$ and $\bicat(\cA, \cK)$ for $Q_j \dashv j$,
 respectively.
\end{prop}
\begin{proof}
  That $U$ is 2-monadic follows because it has a left 2-adjoint given
  by enriched left Kan extension and is furthermore conservative.
  Thus $[\cA, \cK]$ is 2-equivalent to the 2-category of strict
  algebras for $U \circ \, \textrm{Lan}$.  The 2-functor $U$ also has
  a right adjoint given by right Kan extension since $\cK$ is
  complete, so $U \circ \, \textrm{Lan}$ preserves all colimits as it
  is a composite of two left 2-adjoints.  The 2-category $[\textrm{ob}
  \, \cA, \cK]$ is cocomplete since $\cK$ is, hence $T=U \circ \,
  \textrm{Lan}$ satisfies the strongest version of the hypotheses for
  \cref{thm:2monad_str}.  One can check that $\talg$ is 2-equivalent
  to $\gray(\cA, \cK)$ and $\pstalg$ is 2-equivalent to $\bicat(\cA,
  \cK)$ \cite{Lac10Companion}.  This proves that the inclusions $i,j$
  in the statement have left 2-adjoints.  The version of
  \cref{thm:2monad_str} which applies in this case proves, moreover,
  that the components of the units are internal equivalences in
  $\gray(\cA, \cK)$ and $\bicat(\cA, \cK)$, respectively, and hence
  the claim about counits follows (see \cref{rmk:counits}).
\end{proof}

We require one further lemma before stating the main result of this section.
\begin{lem}\label{lem:hom_is_2fun}
  For a fixed 2-category $\cA$, $\bicat(\cA, -)$ is an endo-2-functor
  of the 2-category of 2-categories, 2-functors, and 2-natural
  transformations.
\end{lem}
\begin{proof}
  For any 2-category $\cB$, we know that $\bicat(\cA, \cB)$ is a
  2-category.  Furthermore, if $F \cn \cB \to \cC$ is a 2-functor, it
  is straightforward to check that $F_{*} \cn \bicat(\cA, \cB) \to
  \bicat(\cA, \cC)$ is also a 2-functor.  The only interesting detail
  to check is on the level of 2-cells where we must show that if $\si
  \cn F \Rightarrow G$ is 2-natural, then so is $\si_{*}$.  The
  component of $\si_{*}$ at $H \cn \cA \to \cB$ is the pseudonatural
  transformation $\si H \cn FH \Rightarrow GH$ with $(\si H)_{a} =
  \si_{Ha}$ and similarly for pseudonaturality isomorphisms.  We must
  verify that $\si_{*}$ is 2-natural in $H$.  Thus for any $\al \cn H
  \Rightarrow K$, we must check that $G \al \circ \si H = \si K \circ
  F \al$ as pseudonatural transformations and then similarly for
  modifications.  At an object $a$, we have components
  \[
  (G \al \circ \si H)_{a} 
  = G(\al_{a}) \circ \si_{Ha} 
  = \si_{Ka} \circ F(\al_{a}) 
  = (\si K \circ F \al)_{a}
  \]
  by the 2-naturality of $\si$ in $Ha$.  A short and simple pasting
  diagram argument that we leave to the reader also shows that the
  pseudonaturality isomorphisms for $G \al \circ \si H$ and $\si K
  \circ F \al$ are the same, once again relying on the 2-naturality of
  $\si$ in its argument.  This completes the 1-dimensional part of
  2-naturality, and the 2-dimensional part is a direct consequence of
  the 2-naturality of $\si$ when written out on components.
\end{proof}

\begin{rmk}
  While the argument above is simple, it is not entirely formal.  The
  ``dual'' version for $\bicat(-, \cA)$ does not hold 
  due to an asymmetry in the definition of the pseudonaturality
  isomorphisms for a horizontal composite of pseudonatural
  transformations.
\end{rmk}

We are now ready to prove the main result of this section, namely that
we can replace pseudofunctors $\cA \to \twocatpi$ with equivalent
2-functors $\cA \to \twocattwoi$.
\begin{thm}\label{cor:J_Q_adj}
  The inclusion $J:[\cA, \twocattwoi] \hookrightarrow \bicat(\cA,
  \twocatpi)$ has a left 2-adjoint $Q$.  The unit and counit of this
  adjunction have components which are internal equivalences in
  $\bicat(\cA, \twocatpi)$.
\end{thm}
\begin{proof}
  We will combine \cref{prop:st-left-adj,prop:functor_2monadic}.
  The inclusion $J$ factors into the two inclusions
  \[
  [\cA, \twocattwoi] \stackrel{j}{\hookrightarrow} 
  \bicat(\cA, \twocattwoi) \stackrel{i_*}{\hookrightarrow} 
  \bicat(\cA, \twocatpi).
  \]
  Since $\twocattwoi$ is cocomplete, $j$ has a left 2-adjoint $Q_j$ by
  \cref{prop:functor_2monadic}.  The inclusion $i$ has a left
  2-adjoint $Q_i$ by \cref{prop:st-left-adj}, so $i_*$ has a left
  2-adjoint $(Q_i)_{*}$ by \cref{lem:hom_is_2fun}.  Both of these
  2-adjunctions have units whose components are equivalences, so the
  composite $Q = Q_j(Q_i)_*$ does as well, from which the claim about
  counits follows.
\end{proof}

\section{Categorical suspension models stable suspension}
\label{sec:suspension-models-suspension}

The purpose of this section is to prove
\cref{thm:sigmak_equals_ksigma}, which states that $K$-theory commutes
with suspension, in the appropriate sense. More precisely, we show
that for any permutative category $C$, the $K$-theory spectrum of the
one-object permutative Gray-monoid $\Si C$ is stably equivalent to the
suspension of the $K$-theory spectrum of $C$.

This entails
a comparison between constructions of $K$-theory for categories and
2-categories. Both constructions use the theory of $\Ga$-spaces
developed by Segal \cite{Seg74Categories}.  We recall this theory in \cref{sec:constructions-of-K-theory}.  
Our interest in
$\Ga$-spaces arises from the fact that they model the homotopy theory
of connective spectra, as developed by Bousfield and Friedlander
\cite{BF1978} in the simplicial setting. Thus, in what follows, we
will work with $\Ga$-simplicial sets to prove
\cref{thm:sigmak_equals_ksigma}.

We model the spectra $K(\Si C)$ and $\Si KC$ with $\Ga$-simplicial
sets which are constructed from certain $\Ga$-objects in simplicial
categories.  These $\Ga$-objects in simplicial categories are two
different strictifications of the same pseudofunctor $\sF \to
\Bicat(\De^{\op},\cat_2)$, where $\sF$ is the category of finite
pointed sets and pointed maps.  The first of these strictifications is
provided in \cref{defn:ktheory} by applying the suspension of $\Ga$-simplicial sets
(\cref{defn:suspofgamma}) to a strictification of the pseudofunctor $n \mapsto
C^n$ (\cref{constr:cn}), giving a model for $\Si KC$.  The second is
provided in \cref{defn:naivek} and gives a model for $K(\SI C)$.

In \cref{sec:proof-of-sigma-k-eq-k-sigma} we use the formalism of
\cref{sec:strictification} to compare the two strictifications via a
zigzag of levelwise equivalences.  The key step in this comparison is
constructed in \cref{thm:sigmak_equals_ksigma_full} by strictification of a pseudonatural
equivalence.

\subsection{Constructions of \texorpdfstring{$K$}{K}-theory spectra and suspension}\label{sec:constructions-of-K-theory}

Let $\sF$ denote the following skeletal model for the category of
finite pointed sets and pointed maps. An object of $\sF$ is determined
by an integer $m\geq 0$, which represents the pointed set
$\ul{m}_+=\{0,1,\dots,m\}$, where 0 is the basepoint. This category is
isomorphic to the opposite of the category $\Gamma$ defined by Segal
\cite{Seg74Categories}.

\begin{defn}\label{defn:gammathings}
  Let $\cC$ be a category with a terminal object $\ast$. A
  \emph{$\Gamma$-object} in $\cC$ is a functor $X\cn \sF \to \cC$ such
  that $X(\ulp{0})=\ast$.
\end{defn}  

We give the above definition in full generality, but are only
interested in the cases when $\cC$ is one of $\cat$,
$\iicat$, the category of simplicial sets $\sSet$ or of topological spaces $\Top$.  In each of these cases, we have finite products and a
notion of weak equivalence.  In $\Top$ and $\sSet$ this is the
classical notion of weak homotopy equivalence, and in both $\cat$ and
$\iicat$ we define a functor or 2-functor to be a weak equivalence if
it induces a weak homotopy equivalence in $\sSet$ after applying the
nerve \cite{Gur09Nerves,CCG10Nerves}.

\begin{defn}\label{defn:special}
  Let $X$ be a $\Ga$-object in $\cC$. We say $X$ is \emph{special} if
  the Segal maps 
  \[
  X(\ul{n}_{+}) \to X(\ul{1}_{+})^{n}
  \] 
  are weak equivalences.
\end{defn}

The main result of \cite{Seg74Categories} is that, given a
$\Gamma$-space $X$, one can produce a connective spectrum
$\widetilde{X}$. Moreover, if $X$ is special then $\widetilde{X}$ is
an almost $\Omega$-spectrum such that $\Om^\infty \widetilde{X}$ is a
group completion of $X(\ulp{1})$.  We recall how to express suspension
of spectra in terms of $\Ga$-simplicial sets using the standard
"inclusion" $\De^\op\to \sF$ as specified in \cite[Lemma
3.5]{maythomason} and the following smash product. Let $\wedge \colon
\sF\times \sF\to \sF$ be the functor that sends
$(\ul{n}_{+},\ul{p}_{+})$ to $\ul{(np)}_+=\ul{n}_{+} \vee \ldots \vee
\ul{n}_{+}$. Our reverse lexicographic convention differs from the
smash product in \cite[Construction 3.4]{maythomason} which considers
$\ul{(np)}_+$ as $\ul{p}_{+}\vee \ldots \vee \ul{p}_{+}$.

\begin{notn}\label{notn:bicat_closed}
  Let 
  \[
  \Phi \cn \bicat(\cA \times \cB, \cC) \to \bicat(\cA,\bicat(\cB,\cC))
  \]
  denote the biequivalence of functor bicategories given in
  \cite{street1980fb}, sending a pseudofunctor $F \cn \cA \times \cB \to
  \cC$ to the pseudofunctor
  \[
  \Phi(F)(a)(b) = F(a,b).
  \]
  We also let $\Phi$ denote the isomorphism of functor 2-categories
  \[
  [\cA \times \cB, \cC] \fto{\cong} [\cA, [\cB, \cC]].
  \]
  In order to justify using the same notation $\Phi$ for both of
  these, we note that both versions (reading vertical arrows upwards
  or downwards) of the square below commute,
  \begin{equation}\label{eqn:hom_adjs}
    \begin{xy}
      0;<20mm,0mm>:<0mm,15mm>:: 
      (0,.5)*+{[\cA \times \cB, \cC]}="a";
      (2.5,.5)*+{\bicat(\cA \times \cB, \cC)}="b";
      (0,-.5)*+{[\cA, [ \cB, \cC]]}="c";
      (2.5,-.5)*+{\bicat(\cA, \bicat(\cB, \cC))}="d";
      {\ar@{<->}_{\cong} "a"; "c"};
      {\ar "a"; "b"};
      {\ar "c"; "d"};
      {\ar@{<->}^{\simeq} "b"; "d"};
    \end{xy}
  \end{equation}
  with the downward direction being given by $\Phi$ on the vertical
  arrows.
\end{notn} 

\begin{defn}\label{defn:suspofgamma}
  Let $X\cn \sF \to \sSet$ be a special $\Ga$-simplicial set and let
  $X \circ \sma$ denote the composite
  \[
  \sF \times \De^{\op} \stackrel{\wedge}{\longrightarrow} \sF \stackrel{X}{\longrightarrow} \sSet.
  \]
  Let $d\cn [\De^\op,\sSet] \to \sSet$ denote the diagonal functor.
  We define the \emph{suspension}, $\Si X$, as the special $\Ga$-simplicial set $d \circ \Phi(X
  \circ \sma)$.
\end{defn}

\begin{prop}[\cite{Seg74Categories,BF1978}]\label{prop:suspofgamma}
  Let $X$ be a special $\Ga$-simplicial set and $\widetilde{X}$ its
  associated spectrum. Then the spectrum associated to $\Si X$ is
  stably equivalent to $\Si \widetilde{X}$.
\end{prop}

Given a permutative category $C$, there are several equivalent ways of
constructing a special $\Ga$-category. The following was first
constructed by Thomason \cite[Definition 4.1.2]{Tho79Homotopy}. 

\begin{constr}\label{constr:cn}
  Let $(C, \oplus, e)$ be a permutative category. We can construct a
  pseudofunctor
  \[
  C^{(-)} \cn \sF \to \cat_2
  \]
  which sends $\ul{m}_+$ to $C^{m}$. Given a morphism $\phi \cn
  \ulp{m} \to \ulp{n}$, the corresponding functor $\phi_* \cn C^m \to
  C^n$ is defined uniquely by the requirement that the squares below
  commute for each projection $\pi_j \cn C^n \to C$.
 \[
  \begin{xy}
    0;<25mm,0mm>:<0mm,12mm>:: 
    (0,0)*+{C^m}="a0";
    (0,-1)*+{C^n}="a1";
    (1,0)*+{C^{\phi^{-1}(j)}}="b0";
    (1,-1)*+{C}="b1";
    {\ar "a0"; "b0"};
    {\ar_-{\phi_*} "a0"; "a1"};
    {\ar^-{\oplus} "b0"; "b1"};
    {\ar_-{\pi_j} "a1"; "b1"};
  \end{xy}
  \]
  The top horizontal map is the projection onto the coordinates which
  appear in $\phi^{-1}(j)$. The $\oplus$ appearing on the right
  vertical map is the iterated application of the tensor product
  $\oplus$, with the convention that if $\phi^{-1}(j)$ is empty, then
  the map is the constant functor on the unit $e$. This assignment is
  not strictly functorial, but the permutative structure provides
  natural isomorphisms
  \[
  \psi_* \circ \phi_* \cong (\psi \circ \phi)_*
  \]
  which are uniquely determined by the symmetry.  These isomorphisms
  assemble to make $C^{(-)}$ a pseudofunctor.
\end{constr}
\begin{defn}\label{defn:ktheory}
  The \emph{$K$-theory} of $C$ is the functor 
  \[
  KC=N\circ Q_j\big(C^{(-)}\big)\cn \sF \to \sSet,
  \] 
  where $N$ is the usual nerve functor $\cat \to \sSet$ and $Q_j$ is
  the left 2-adjoint from \cref{prop:functor_2monadic} when $\cK =
  \cat_2$.
\end{defn} 

\begin{rmk}\label{rmk:reduced}
  Although the pseudofunctor $C^{(-)}$ satisfies the property that it
  maps $\ulp{0}$ to $\ast$, its strictification $Q_j\big(C^{(-)}\big)$
  does not. Thus $Q_j\big(C^{(-)}\big)$ is a functor $\sF \to \cat$,
  but it is not a $\Ga$-category as in \cref{defn:gammathings}. Since
  $Q_j\big(C^{(-)}\big)$ is levelwise equivalent to $C^{(-)}$, and in
  particular, $Q_j\big(C^{(-)}\big)(\ulp{0})$ is contractible, we can
  replace $N\circ Q_j\big(C^{(-)}\big)$ by a levelwise equivalent
  $\Ga$-simplicial set. This replacement is made implicitly here, and
  throughout the remainder of the paper.
\end{rmk}

\begin{lem}\label{lem:adj_doesnt_matter}
  Consider the composite
  \[
  [\sF \times \De^{\op}, \cat] \xrightarrow{\ \Phi\ } [\sF, [\De^{\op}, \cat]] \xrightarrow{\ N_{*} \circ\, -\ } [\sF, [\De^{\op}, \sset]] \xrightarrow{\ d \circ\, -\ } [\sF,\sSet].
  \]
  If $F$ is a levelwise weak equivalence of diagrams $\sF \times
  \De^{\op} \to \cat$, then $dN_*\Phi(F)$ is a levelwise weak
  equivalence of diagrams $\sF \to \sSet$.
\end{lem}
\begin{proof}
  This follows from \cite[Theorem B.2]{BF1978}, which states that if
  $f\cn X\to Y$ is a map of bisimplicial sets such that
  $X_{n,\bullet}\to Y_{n,\bullet}$ is a weak equivalence of simplicial
  sets for all $n \geq 0$, then $d(f)\cn d(X) \to d(Y)$ is a weak
  equivalence.
\end{proof}

To relate the $\Gamma$-simplicial set $\Sigma KC$ to the $K$-theory of the permutative Gray-monoid $\Sigma C$, we provide a new construction of a special $\Ga$-2-category $\ul{K}(\Si C)$ and show it is levelwise weakly equivalent to the $K$-theory defined in \cite{GJO2015KTheory}. 
\begin{notn}\label{notn:tricat_twocatpim}
  Let $\twocatppm$ denote the tricategory whose objects are
  2-categories, and whose higher cells are pseudofunctors,
  pseudonatural transformations, and modifications
  \cite{Gurski13Coherence}.
\end{notn}

\begin{lem}\label{lem:pgm_to_psf}
  Let $(\cD, \oplus, e)$ be a permutative Gray-monoid.  Then there is
  a pseudofunctor of tricategories $\cD^{(-)} \cn \sF \to \twocatppm$
  with value at $\ul{m}_{+}$ given by $\cD^{m}$.  If $\cD$ has a
  single object, then this becomes a pseudofunctor of 2-categories
  $\cD^{(-)} \cn \sF \to \twocatpi$.
\end{lem}

\begin{proof}
  The first claim is a special case of \cite[Theorem
  2.5]{GO12Infinite}. For the second claim, by \cref{cor:equivs}, it
  suffices to work with $\Si D$ for a permutative category $D$. Recall
  from \cref{constr:cn} that we have the pseudofunctor
  \[
  D^{(-)} \cn \sF \to \Cat_2.
  \]
  The permutative structure on $D$ in fact makes each $D^{m}$ a strict
  monoidal category with pointwise tensor product and unit, and each
  functor $\phi_* \cn D^{m} \to D^{n}$ for $\phi \cn \ul{m}_{+} \to
  \ul{n}_{+}$ a strong monoidal functor.  One can verify that the
  isomorphisms $\psi_* \circ \phi_* \cong (\psi \circ \phi)_*$ are
  themselves monoidal, so we get a pseudofunctor
  \[
  D^{(-)} \cn \sF \to \stmoncat
  \]
  from $\sF$ to the 2-category $\stmoncat$ of strict monoidal categories, strong
  monoidal functors, and monoidal natural transformations.  Note that
  $(\Si D)^{m} \cong \Si(D^{m})$, so we define
  \[
  (\Si D)^{(-)} = \Si \circ  D^{(-)}
  \]
  where $\Si$ is now the 2-functor $\stmoncat \to \twocatpi$ which
  views each strict monoidal category as the hom-category of a
  2-category with a single object.  This composite is the desired
  pseudofunctor.
\end{proof}

\begin{defn}[\cite{LP20082nerves}]\label{defn:2nerve}
  Let $\cA$ be a 2-category.  The \emph{nerve} of $\cA$ is the
  simplicial category $N\cA \cn \De^{\op} \to \Cat$ defined by
  \[
  N\cA_{n} = \twocattwoi([n], \cA)
  \]
  where $[n]$ is the standard category $0 \to 1 \to \cdots \to n$
  treated as a discrete 2-category.
  This is the function on objects of a 2-functor from $\twocattwoi$ to $[\De^\op,\cat_2]$.
\end{defn}

\begin{rmk}\label{rmk:nerve_or_2nerve}
  We note that this is called the \emph{2-nerve} by Lack and Paoli.
  It is related but not equal to the general bicategorical nerve of
  \cite{Gur09Nerves,CCG10Nerves}.  Detailed comparisons are given in
  \cite{CCG10Nerves}.
\end{rmk}

Unpacking this definition, $N\cA_0=\ob \cA$ as a discrete
category. When $n\geq 1$,
\[
N\cA_{n}=\coprod_{a_0,\dots, a_n\in \ob \cA} \cA(a_{n-1},a_n)\times\cdots \times \cA(a_0,a_1).
\]
Using this same formula, we define the nerve on $\catgrph$ which fits in the following commuting diagram.
\[
\begin{xy}
  0;<28mm,0mm>:<0mm,6mm>:: 
  (0,0)*+{\Cat\mh\mathpzc{Grph}}="a0";
  (0,2)*+{\twocattwoi}="a2";
  (1,0)*+{[\ob\, \Delta^{\op}, \Cat_2]}="b0";
  (1,2)*+{[\Delta^{\op}, \Cat_2]}="b2";
  {\ar_-{N} "a0"; "b0"};
  {\ar^-{N} "a2"; "b2"};
  {\ar_-{} "a2"; "a0"};
  {\ar^-{} "b2"; "b0"};
\end{xy}
\]
Let $S$ be the 2-monad on $\Cat\mh\mathpzc{Grph}$ whose algebra
2-category is $\twocattwoi$ (\cref{prop:2cat2i_2monadic}).  Let $T$
be the 2-monad on $[\ob\, \Delta^{\op}, \Cat_2]$ whose algebra 2-category
is $[\Delta^{\op}, \Cat_2]$ (\cref{prop:functor_2monadic}).  
We now apply \cref{prop:map=lift} to show that the nerve extends to $\twocatpi$.
\begin{lem}\label{lem:nerve-is-map-of-2-monads}
  The nerve
  $N$ is a strict map of 2-monads $S \to T$ and therefore provides the
  middle map in the
  commutative diagram below.
  \[
  \begin{xy}
    0;<28mm,0mm>:<0mm,12mm>:: 
    (0,0)*+{\Cat\mh\mathpzc{Grph}}="a0";
    (0,1)*+{\twocatpi}="a1";
    (0,2)*+{\twocattwoi}="a2";
    (1,0)*+{[\ob\, \Delta^{\op}, \Cat_2]}="b0";
    (1,1)*+{\gray(\Delta^{\op}, \Cat_2)}="b1";
    (1,2)*+{[\Delta^{\op}, \Cat_2]}="b2";
    {\ar_-{N} "a0"; "b0"};
    {\ar^-{N} "a1"; "b1"};
    {\ar^-{N} "a2"; "b2"};
    {\ar_-{i} "a2"; "a1"};
    {\ar_-{U} "a1"; "a0"};
    {\ar^-{i} "b2"; "b1"};
    {\ar^-{U} "b1"; "b0"};
  \end{xy}
  \]
\end{lem}

We now define the $\Ga$-objects we will use to understand $K$-theory
of a suspension.

\begin{defn}\label{defn:naivek}
  Let $C$ be a permutative category with $\Si C$ its suspension
  permutative Gray-monoid.  Let $Q=Q_{j}(Q_i)_{*}$ denote the left 2-adjoint of the
  inclusion $J:[\sF, \twocattwoi] \hookrightarrow \bicat(\sF,
  \twocatpi)$ constructed in \cref{cor:J_Q_adj}.
\begin{enumerate}
\item Define $\ul{K} (\Si C)$ to be $Q\Big( (\Si
  C)^{(-)}\Big)$.  This is a functor $\sF \to  \iicat$. 
\item\label{it:Kadj} The composite $N\circ \ul{K}(\Si C)$ is a functor $\sF\to [\Delta^{\op},\cat]$. Define $K_{adj} (\Si C)$ to be $\Phi^{-1} \big( N\circ \ul{K} (\Si C) \big)$.
\end{enumerate}
\end{defn}

The composite
\[\iicat \fto{N} [\Delta^{\op},\cat] \fto{N_*} [\Delta^{\op},\sSet]\fto{d} \sSet\]
is one of the versions of the nerve for 2-categories in
\cite{CCG10Nerves}. Post-composing $\ul{K}(\Si C)$ with this functor
(and, as noted in \cref{rmk:reduced}, implicitly replacing with a
reduced diagram) yields a $\Ga$-simplicial set which is a model of the
$K$-theory of $\Si C$. We make this rigorous in the following lemma,
which relates the definition of $K$-theory here with that introduced
in \cite{GJO2015KTheory}, here denoted by $\widetilde{K}$.

For a permutative Gray-monoid $\cD$, $\widetilde{K}(\cD)$ is a special
$\Ga$-2-category such that an object at level $n$ is an object in
$\cD$, together with an explicit way of decomposing it as a sum of $n$
objects. This allows for strict functoriality with respect to
$\sF$. This construction generalizes the construction of
\cite{may78spectra,Man10Inverse} for permutative categories.

\begin{lem}\label{naivek_equals_k}
  Let $(C, \oplus, e)$ be a permutative category.  There is a levelwise weak
  equivalence between the $\Ga$-2-categories
  $\ul{K}(\Si C)$ and $\widetilde{K}(\Si C)$, hence a stable
  equivalence between the spectra these represent.
\end{lem}
\begin{proof}
  We shall prove that there is a levelwise weak equivalence $\ul{K}(\Si C) \to \widetilde{K}(\Si C)$ of $\Ga$-2-categories.  Since both of these are special, it suffices to construct such a map
  and check that it is a weak equivalence when evaluated at
  $\ul{1}_{+}$.  The functor $Q$ is a left adjoint, so strict maps
  $Z\cn \ul{K}(\Si C)=Q\Big( (\Si C)^{(-)}\Big) \to \widetilde{K}(\Si C)$ are in
  bijection with pseudonatural transformations
  \[
  \check{Z} \cn (\Si C)^{(-)} \to \widetilde{K}(\Si C)
  \]
  in $\bicat(\sF, \twocatpi)$.  This bijection is induced by
  composition with a universal pseudonatural transformation
  $\eta\cn(\Si C)^{(-)} \to Q\Big( (\Si C)^{(-)}\Big)$, so we have the
  commutative triangle shown below.
  \[
  \xy
  (0,0)*+{(\Si C)^{(-)}}="1"; 
  (25,0)*+{Q\Big( (\Si C)^{(-)}\Big)}="2";
  (25,-15)*+{\widetilde{K}(\Si C)}="3";
  {\ar^{\eta} "1";"2" };
  {\ar^{{Z}} "2";"3"  };
  {\ar_{\check{Z}} "1";"3"  };
  \endxy
  \]
  We know that $\eta$ is a levelwise weak equivalence by
  \cref{cor:J_Q_adj}, so the component of $Z$ at $\ul{1}_{+}$ is a
  weak equivalence if and only if the same holds for $\check{Z}$.

  We will construct the pseudonatural transformation $\check{Z}$.  In order to do so,
  we briefly review the data which define the cells of
  $\widetilde{K}(\Si C)(\ul{n}_{+})$; we omit the axioms these data
  must satisfy and refer the reader to \cite{GJO2015KTheory}.  Because
  $\Si C$ has a single object, an object of $\widetilde{K}(\Si
  C)(\ul{n}_{+})$ consists of objects $c_{s,t}$ of the permutative
  category $C$ for $s,t$ disjoint subsets of $\ul{n} = \{1, \ldots, n
  \}$.  We denote such an object as $\{c_{s,t}\}$ or, when more detail
  is useful, a function
  \[
  \{s, t \,\mapsto\, c_{s,t}\}.
  \]
  A 1-cell $\{ c_{s,t} \} \to \{ d_{s,t} \}$ consists of objects $x_s$
  of $C$ for $s \subset \ul{n}$ together with isomorphisms
  \[
  \ga_{s,t} \cn x_t \oplus x_s \oplus c_{s,t} \cong d_{s,t} \oplus x_{s \cup t}.
  \]
  We denote this as $\{x_s, \ga_{s,t}\}$ or, in functional notation,
  \[
  \left\{ \genfrac{}{}{0pt}{}{s \mapsto x_s}{s,t \mapsto \ga_{s,t}}  \right\}.
  \]
  A 2-cell $\{x_s, \ga_{s,t} \} \Rightarrow \{ y_s, \de_{s,t} \}$
  consists of morphisms $\al_s \cn x_s \to y_s$ in $C$.  We denote
  this $\{\al_s\}$ or with a corresponding functional notation.

  Now $(\Si C)^{\ul{n}_{+}}$ is $(\Si C)^{n} \cong \Si (C^{n})$ by
  definition.  We define $\check{Z}$ on cells as follows.
  \begin{itemize}
  \item The unique 0-cell of $\Si (C^{n})$ maps to the object of
    $\widetilde{K}(\Si C)(\ul{n}_{+})$ with $c_{s,t} = e$ for all
    $s,t$.
  \item A 1-cell $(x_1, \ldots, x_n)$ maps to the 1-cell
    \[
    \left\{ 
      \genfrac{}{}{0pt}{}
      {s \mapsto \oplus_{i\in s} x_i}
      {s,t \mapsto \la_{s,t}}
    \right\}.  
    \]
    where $\la_{s,t}$ denotes the unique interleaving symmetry isomorphism
    \[
    (\oplus_{i\in s} x_i) \oplus (\oplus_{j\in t} x_j) \cong \oplus_{k\in s \cup t} x_k.
    \]
  \item A 2-cell $(f_1, \ldots, f_n)$ maps to the 2-cell 
    \[
    \{s \mapsto \oplus_{i\in s} f_i\}.
    \]
  \end{itemize}
  Using the permutative structure of $C$, it is straightforward to
  verify that the formulas above satisfy the axioms of \cite[Section
  6.1]{GJO2015KTheory} and therefore define valid cells.  Clearly
  $\check{Z}$ sends the identity 1-cell of $\Si (C^{n})$, namely $(e,
  \ldots, e)$, to the identity 1-cell in $\widetilde{K}(\Si
  C)(\ul{n}_{+})$.  Now composition of 1-cells in $\Si (C^{n})$ is
  given by the monoidal structure, so
  \[
  (x_1, \ldots, x_n) \circ (y_1, \ldots, y_n) = (x_1 \oplus y_1, \ldots, x_n \oplus y_n).
  \]
  We have a similar formula for composition in $\widetilde{K}(\Si
  C)(\ul{n}_{+})$, with the object part of $\{x_s, \ga_{s,t} \} \circ
  \{y_s, \de_{s,t} \}$ being given on $s$ by $x_s \oplus y_s$.  From
  these formulas, we see that $\check{Z}$ does not strictly preserve
  1-cell composition since
  \[
  \check{Z}(x_1, \ldots, x_n) \circ \check{Z}(y_1, \ldots, y_n) = 
  \left\{ 
    \genfrac{}{}{0pt}{}
    {s \mapsto (\oplus_{i\in s} x_i) \oplus (\oplus_{i\in s} y_i)}
    {s,t \mapsto \mu_{s,t}}
  \right\}
  \]
  where $\mu$ denotes the unique interleaving symmetry isomorphism.
  On the other hand,
  \[
  \check{Z}(x_1\oplus y_1, \ldots, x_n\oplus y_n) = 
  \left\{ 
    \genfrac{}{}{0pt}{}
    {s \mapsto \oplus_{i \in s} (x_i \oplus y_i)}
    {s,t \mapsto \la_{s,t}}
  \right\}.
  \]
  These are isomorphic by a unique
  symmetry, and that data equips
  \[
  \check{Z}(\ul{n}_{+}) \cn (\Si C)^{{n}} \to \widetilde{K}(\Si C)(\ul{n}_{+})
  \]
  with the structure of a normal (i.e., strictly unit-preserving)
  pseudofunctor.
  
  Now let $\phi \cn \ul{m}_{+} \to \ul{n}_{+}$ in $\sF$.  We must
  construct an invertible icon in the square below.
  \begin{equation*}
    \begin{xy}
      0;<13mm,0mm>:<0mm,15mm>:: 
      (0,.5)*+{(\Si C)^{m}}="a";
      (2.5,.5)*+{\widetilde{K}(\Si C)(\ul{m}_{+})}="b";
      (0,-.5)*+{(\Si C)^{n}}="c";
      (2.5,-.5)*+{\widetilde{K}(\Si C)(\ul{n}_{+})}="d";
      {\ar@{->}_{\phi_*} "a"; "c"};
      {\ar^{\check{Z}} "a"; "b"};
      {\ar_{\check{Z}} "c"; "d"};
      {\ar@{->}^{\phi_*} "b"; "d"};
    \end{xy}
  \end{equation*}
  We begin by noting that this diagram obviously commutes on the
  unique object, so there can exist an icon (see \cref{defn:icon})
  between the two composite pseudofunctors.  The top and right
  composite sends a 1-cell $(x_1, \ldots, x_n)$ to the 1-cell
  \[
  \left\{ 
    \genfrac{}{}{0pt}{}
    {u \mapsto \oplus_{i \in \phi^{-1}(u)} x_i}
    {u,v \mapsto \lambda_{\phi^{-1}(u), \phi^{-1}(v)}}
  \right\}
  \]
  The left and bottom composite then sends $(x_1, \ldots, x_n)$ to the
  1-cell with
  \[
  \left\{ 
    \genfrac{}{}{0pt}{}
    {u \mapsto \oplus_{i \in u} \big(\oplus_{j \in \phi^{-1}(i)} x_j \big)}
    {u,v \mapsto \kappa_{u,v}}
  \right\}
  \]
  where $\kappa_{u,v}$ interleaves the blocks $\big(\oplus_{j \in \phi^{-1}(i)} x_j \big)$.  
  
  There is an invertible 2-cell between these 1-cells which is given
  by the symmetry isomorphism
  \[
  \oplus_{i \in \phi^{-1}(u)} x_i \cong \oplus_{i \in u} \big(\oplus_{j \in \phi^{-1}(i)} x_j \big).
  \]
  Coherence for symmetric monoidal categories, together with the
  naturality of symmetries, implies that the icon axioms hold.
  Further, the same coherence shows that these invertible icons are
  themselves the naturality isomorphisms which constitute a
  pseudonatural transformation between pseudofunctors $\sF \to
  \twocatpi$.
  
  Our final task is to verify that $\check{Z}(\ul{1}_{+})$ is a weak
  equivalence.  It is a simple calculation to check that in fact
  $\check{Z}(\ul{1}_{+})$ induces an isomorphism of 2-categories
  $\widetilde{K}(\Si C)(\ul{1}_{+}) \cong \Si C$.
\end{proof}

\begin{rmk}
  One can check that the equivalence constructed in \cref{naivek_equals_k}
  is pseudonatural in the variable $C$.
\end{rmk}

\subsection{Proof of Theorem
  \ref{thm:sigmak_equals_ksigma}}\label{sec:proof-of-sigma-k-eq-k-sigma}

Given a permutative category $C$, we can construct two
pseudofunctors from $\sF$ to  $\bicat(\De^{\op}, \cat_2)$.
One is the composite
\[
\sF \fto{(\Si C)^{(-)}} \twocatpi \fto{N} \gray(\De^{\op}, \cat_2)\hookrightarrow \bicat(\De^{\op},\cat_2),
\]
where $N$ denotes the nerve functor of
\cref{lem:nerve-is-map-of-2-monads}.  The other is given by
$\Phi(C^{(-)}\circ \sma)$, where
\[
\Phi: \bicat(\sF\times \De^{\op}, \cat_2) \to \bicat(\sF,\bicat(\De^{\op},\cat_2))
\]
is the 2-functor from \cref{notn:bicat_closed}
and $C^{(-)} \circ \sma$
is the composite
\[
\sF \times \De^{\op} \stackrel{\wedge}{\longrightarrow} \sF \stackrel{C^{(-)}}{\longrightarrow} \cat_2.
\]

\begin{prop}\label{prop:psfuns_equal}
  With notation as above, $\Phi(C^{(-)} \circ \sma) = N\circ (\Si
  C)^{(-)}$.
\end{prop}
\begin{proof}
  This result follows from a direct comparison of $\Phi(C^{(-)} \circ
  \sma)$ with $N\circ(\Sigma C)^{(-)}$.  Both pseudofunctors send the
  object $\ulp{m}$ in $\sF$ to the 2-functor $\De^\op \to \Cat_2$ given by
  \begin{align*}
    [p] \mapsto & C^{m \cdot p} = (C^m)^p\\
    ([p]\fto{\al} [q]) \mapsto & \big(C^{m \cdot p} \fto{(m \sma \al)_*} C^{m \cdot q}\big).
  \end{align*}
  For $\Phi(C^{(-)} \circ \sma)$ this is immediate.  For
  $N\circ(\Sigma C)^{(-)}$ this follows because $\Sigma C$ has only
  one object and the horizontal composition of cells is given by the
  monoidal product in $C$.

  Both pseudofunctors send a morphism $\phi \cn \ulp{m} \to \ulp{n}$ in
  $\sF$ to the pseudonatural transformation whose component at $[p]
  \in \De^{\op}$ is given by
  \[
  C^{m \cdot p} \fto{(\phi \sma p)_*} C^{n \cdot p}.
  \]
  
  For $\Phi(C^{(-)} \circ \sma)$ it is immediate that the
  pseudonaturality constraint has components given by
  \begin{equation}\label{eq:two-cell-constraint}
    (\ulp{n} \sma \al)_* \circ (\phi \sma [p])_* \iso (\phi \sma \al)_* \iso (\phi \sma [q])_* \circ (\ulp{m} \sma \al)_*
  \end{equation}
  at $\al \cn [p] \to [q]$.  These isomorphisms are the
  pseudofunctoriality constraints of $C^{(-)}$ and are instances of
  the symmetry in $C$ (see \cref{constr:cn}).  A straightforward check
  shows that the pseudofunctoriality constraint of $N\circ(\Sigma
  C)^{(-)}$ is given by the same instances of the symmetry of $C$.

  For a composable pair $\phi\cn \ulp{m} \to \ulp{n}$ and $\psi\cn
  \ulp{n} \to \ulp{k}$, the symmetry of $C$ provides 
  \[
  (\psi \sma [p])_* \circ (\phi \sma [p])_* \cong 
  \big((\psi \circ \phi) \sma [p]\big)_*
  \]
  and these are the components of the pseudofunctoriality of
  $\Phi(C^{(-)} \circ \sma)$.  The same computation holds for $N\circ(\Sigma C)^{(-)}$.
\end{proof}

We are now ready for the main theorem of this section, from which the
proof of \cref{thm:sigmak_equals_ksigma} follows.  Let $Q_j$ be as in
\cref{defn:ktheory}: the left 2-adjoint to the inclusion functor 
\[
j \cn [\sF \times \De^{\op}, \cat_2] \hookrightarrow \bicat(\sF
  \times \De^{\op}, \cat_2).
\]

\begin{thm}\label{thm:sigmak_equals_ksigma_full}
  For any permutative category $C$, there is a zigzag of levelwise
  equivalences between $Q_{j}(C^{(-)}) \circ \wedge$ and $K_{adj}(\Si C)$. 
\end{thm}
\begin{proof} 
  The components of the
  unit and counit of the 2-adjunction $Q_j \dashv j$ are internal equivalences in
  $\bicat(\sF \times \De^{\op}, \cat_2)$ by
  \cref{prop:functor_2monadic}.  Assume that
  \[
  \al \cn j\Big(  Q_{j}(C^{(-)}) \circ \wedge \Big) \stackrel{\simeq}{\longrightarrow} j\Big(K_{adj}(\Si C)\Big) 
  \]
  is a pseudonatural equivalence in $\bicat(\sF \times \De^{\op},
  \cat_2)$.  Since a pseudonatural equivalence is an internal
  equivalence in $\bicat(\sF \times \De^{\op}, \cat_2)$, we can apply
  $Q_j$ and get an internal equivalence in $[\sF \times \De^{\op},
  \cat_2]$.  This gives a zigzag
  \[
  Q_{j}(C^{(-)}) \circ \wedge \stackrel{\epz}{\longleftarrow} Q_j j\Big(  Q_{j}(C^{(-)}) \circ \wedge \Big) \xrightarrow{\ Q_j(\al)\ } Q_j j\Big(K_{adj}(\Si C) \Big) \stackrel{\epz}{\longrightarrow} K_{adj}(\Si C)
  \]
  in $[\sF \times \De^{\op}, \cat_2]$ in which the first and third
  arrows are levelwise equivalences as they are internal equivalences
  in $\bicat(\sF \times \De^{\op}, \cat_2)$, and the second arrow is a
  levelwise equivalence as it is an internal equivalence (i.e.,
  2-equivalence) in $[\sF \times \De^{\op}, \cat_2]$.  It only remains
  to construct an equivalence $\al$ as above.

  In order to construct the pseudonatural equivalence $\al$, first
  recall from \cref{defn:naivek} (\ref{it:Kadj}) that
  \[
  K_{adj}(\Si C) = \Phi^{-1}\Big( N \circ Q\big((\Si C)^{(-)}\big) \Big)
  \]
  where $\Phi$ denotes the adjunction of \cref{notn:bicat_closed} and 
  $Q$ denotes the left adjoint constructed in \cref{cor:J_Q_adj}.
  We define $\al$ as the composite below,
  which we explain afterwards. 
   \[
   \begin{array}{rcl}
   j \big( Q_{j}(C^{(-)}) \circ \wedge \big) & \stackrel{=}{\longrightarrow} & jQ_j\big(C^{(-)}\big) \circ \wedge \\
   & \stackrel{\simeq}{\longrightarrow} & C^{(-)} \circ \wedge \\
   & \stackrel{\simeq}{\longrightarrow} & \Phi^{-1}\Big( N \circ (\Si C)^{(-)} \Big) \\
   & \stackrel{\simeq}{\longrightarrow} & \Phi^{-1}\Big( N \circ JQ\big((\Si C)^{(-)}\big) \Big) \\
   & \stackrel{=}{\longrightarrow} & j\Phi^{-1}\Big( N \circ Q\big((\Si C)^{(-)}\big) \Big) \\
   & \stackrel{=}{\longrightarrow} & j K_{adj}(\Si C)
   \end{array}
   \]
   The equality giving the first arrow is a simple calculation.  
   The
   equivalence giving the second arrow is a pseudo-inverse of the unit
   for $Q_j \dashv j$, whiskered by $\wedge$ and hence still an
   equivalence.  The equivalence giving the third arrow is the adjoint
   of the equality in \cref{prop:psfuns_equal}.  The equivalence
   giving the fourth arrow is derived from the unit of $Q \dashv J$
   which is itself an equivalence, so whiskering with $N$ and applying
   $\Phi^{-1}$ still yields an equivalence.  The equality giving the
   fifth arrow follows from the commutativity of \cref{eqn:hom_adjs},
   and the equality giving the final arrow is \cref{defn:naivek}
   (\ref{it:Kadj}).
\end{proof}

\begin{rmk}\label{rmk:naturality}
  The zigzag in \cref{thm:sigmak_equals_ksigma} is natural up to
  homotopy.  More precisely, this zigzag consists of three maps, two
  of which are counits for the 2-adjunction $Q_j \dashv j$.  It is
  easy to see that $C \mapsto C^{(-)}$ sends symmetric, strong
  monoidal functors between permutative categories to pseudonatural
  transformations between their corresponding pseudofunctors $\sF \to \cat_2$,
  so a symmetric, strong monoidal functor $F \cn C \to D$ will yield a
  2-natural transformation
  \[
  Q_{j}(C^{(-)}) \circ \wedge \to Q_{j}(D^{(-)}) \circ \wedge.
  \]
  The counit $\epz$ is strictly natural with respect to such, so the
  first map in our zigzag is strictly natural in symmetric, strong
  monoidal functors. A similar argument holds for $K_{adj}$, so the
  third map in our zigzag is also strictly natural in symmetric,
  strong monoidal functors.  The second map is what is called
  $Q_j(\al)$ in the proof above.  It is more involved, but a careful
  check reveals that each of the maps of which it is a composite is
  pseudonatural in symmetric, strong monoidal functors, and so the
  same will be true after applying $Q_j$.  Thus our zigzag is actually
  pseudonatural in the variable $C$, which in particular implies that
  it is natural up to homotopy when viewed as a zigzag of spectra.
\end{rmk}

\begin{proof}[Proof of \cref{thm:sigmak_equals_ksigma}]
  On one hand, the suspension of $\Ga$-simplicial sets given in
  \cref{defn:suspofgamma} models the stable suspension by
  \cref{prop:suspofgamma}.  Recalling
  \cite{Tho79Homotopy,maythomason}, the $\Gamma$-simplicial set
  $KC=N\circ Q_j (C^{(-)})$ from \cref{defn:ktheory} models the
  $K$-theory spectrum of $C$.  Its suspension as a $\Ga$-simplicial
  set, $\Si K(C)$, is given by composing the diagonal $d$ with
  $\Phi(K(C) \circ \sma)$.  By naturality of $\Phi$ in its target
  2-category, this is given by $d N_* \Phi(Q_j (C^{(-)}) \circ
  \wedge)$.  By \cref{lem:adj_doesnt_matter}, a levelwise weak
  equivalence of functors $X,Y \cn \sF \times \De^\op \to \cat_2$
  induces a levelwise weak equivalence between $dN_*\Phi (X)$ and
  $dN_*\Phi (Y)$.  Therefore it suffices to examine $Q_j (C^{(-)})
  \circ \sma$.  On the other hand, in \cref{defn:naivek} we have the $\Ga$-2-category
  $\ul{K}(\Si C) = Q\big( (\Si C)^{(-)} \big)$
  and the related adjoint $K_{adj}(\Si C) = \Phi^{-1} \big(N \circ \ul{K}(\Si C)\big)$.  \Cref{naivek_equals_k} shows that $d N_* \Phi(K_{adj}(\Si
  C))$ models the $K$-theory spectrum of $\Si C$.  Finally, the result
  follows by \cref{thm:sigmak_equals_ksigma_full}, which shows that
  there is a zigzag of levelwise equivalences between $Q_j (C^{(-)})
  \circ \sma$ and $K_{adj}(\Si C)$.
\end{proof}

\bibliographystyle{amsalpha2}
\bibliography{nostrskelmod}%

\end{document}